\documentclass[10pt]{amsart}
\usepackage{fancyhdr}
\usepackage{stmaryrd}
\usepackage{calrsfs}

\usepackage{amsmath}
\usepackage[nospace,noadjust]{cite}
\usepackage{amsfonts}
\usepackage{amssymb,enumerate}
\usepackage{amsthm}
\usepackage{cite}
\usepackage{comment}
\usepackage{color}
\usepackage[all]{xy}
\usepackage{hyperref}
\usepackage{lineno}

\newtheorem*{maintheorem*}{Main Theorem}
\newtheorem{theorem}{Theorem}[section]
\newtheorem{prop}[theorem]{Proposition}
\newtheorem{question}[theorem]{Question}

\newtheorem{lemma}[theorem]{Lemma}
\newtheorem{cor}[theorem]{Corollary}

\theoremstyle{definition}
\newtheorem{definition}[theorem]{Definition}

\newtheorem{example}[theorem]{Example}
\numberwithin{equation}{section}

\newcommand{\cc}{\mathbb{C}}
\newcommand{\ff}{\mathbb{F}}
\newcommand{\nn}{\mathbb{N}}
\newcommand{\qq}{\mathbb{Q}}
\newcommand{\rr}{\mathbb{R}}
\newcommand{\zz}{\mathbb{Z}}

\newcommand{\gp}{\text{gp}}

\newcommand{\pp}{\mathbb{P}}

\newcommand{\ii}{\mathcal{A}}

\newcommand{\uu}{\mathcal{U}}

\providecommand\ldb{\llbracket}
\providecommand\rdb{\rrbracket}

\hypersetup{
	pdftitle={FD},
	colorlinks=true,
	linkcolor=blue,
	citecolor=cyan,
	urlcolor=wine
}



\begin{document}
	
	\mbox{}
	\title{Integral domains and the IDF property}
	
	\author{Felix Gotti}
	\address{Department of Mathematics\\MIT\\Cambridge, MA 02139}
	\email{fgotti@mit.edu}
	
	\author{Muhammad Zafrullah}
	\address{Department of Mathematics\\Idaho State University\\Pocatello, ID 83209}
	\email{mzafrullah@usa.net}
	
\date{\today}

\begin{abstract}
	 An integral domain $D$ is called an irreducible-divisor-finite domain (IDF-domain) if every nonzero element of $D$ has finitely many irreducible divisors up to associates. The study of IDF-domains dates back to the seventies. In this paper, we investigate various aspects of the IDF property. In 2009, P.~Malcolmson and F. Okoh proved that the IDF property does not ascend from integral domains to their corresponding polynomial rings, answering a question posed by D. D. Anderson, D. F. Anderson, and the second author two decades before. Here we prove that the IDF property ascends in the class of PSP-domains, generalizing the known result (also by Malcolmson and Okoh) that the IDF property ascends in the class of GCD-domains. We put special emphasis on IDF-domains where every nonunit is divisible by an irreducible, which we call TIDF-domains, and we also consider PIDF-domains, which form a special class of IDF-domains introduced by Malcolmson and Okoh in 2006. We investigate both the TIDF and the PIDF properties under taking polynomial rings and localizations. We also delve into their behavior under monoid domain and $D+M$ constructions.
\end{abstract}
\medskip

\dedicatory{Dedicated to the memory of Robert Gilmer, \\ the most influential figure in Multiplicative Ideal Theory in recent past.}

\maketitle


\bigskip
\section{Introduction}
\label{sec:intro}

Following A. Grams and H. Warner~\cite{GW75}, we say that an integral domain $D$ is an IDF-domain provided that every nonzero element of $D$ has only finitely many irreducible divisors up to associates. Back in 1990, D. D. Anderson, D. F. Anderson, and the second author characterized atomic IDF-domains as integral domains with elements having only finitely many factorizations into irreducibles \cite[Theorem~5.1]{AAZ90}, and they called such domains finite factorization domains (or FFDs). Although they proved that the finite factorization property ascends to polynomial rings, it was not clear by then whether the IDF property (or the atomic property) ascended to polynomial rings. The corresponding question was posed as \cite[Question~2]{AAZ90}, and it was not until two decades later that P. Malcolmson and F. Okoh \cite[Theorem~2.5]{MO09} provided a negative answer. The fundamental purpose of this paper is to investigate whether the IDF property (and some stronger properties) are preserved under various standard algebraic constructions, including polynomial rings, $D+M$ constructions, and localizations.
\smallskip

Although, in general, the IDF property does not ascend to polynomial rings, there are certain relevant classes of integral domains where the IDF property ascends. For instance, if a GCD-domain $D$ is an IDF-domain, then $D[X]$ is an IDF-domain (and a GCD-domain); this was proved in \cite[Theorem~1.9]{MO09} (the GCD part is a special case of \cite[Theorem~6.4]{GP74}). More general, the second author proved in~\cite{mZ17} that the IDF property ascends in the class of pre-Schreier domains (although this result was never published). An integral domain where every nonempty finite subset of nonzero elements has finitely many maximal common divisors up to associates is called an MCD-finite domain. It is clear that every GCD-domain is an MCD-finite domain, and it is not hard to verify that every pre-Schreier domain is an MCD-finite domain. More recently, S. Eftekhari and M.~R. Khorsandi \cite[Theorem~2.1]{EK18} proved that the IDF property ascends in the class of MCD-finite domains.
\smallskip

Following J. T. Arnold and P. B. Sheldon \cite{AS75}, we say that an integral domain $D$ is a PSP-domain if every primitive polynomial $f$ over $D$ is super-primitive; that is, if the ideal $A_f$ generated by the coefficients of $f$ is contained in no proper principal ideal of $D$, then $A_f^{-1} = D$. In Section~\ref{sec:UTZ ideals and ascending of IDF in PSP and pre-Schreier domains}, we prove that if a PSP-domain $D$ satisfies the IDF property, then $D[X]$ also satisfies the IDF property. This result improves upon the fact that the IDF property ascends in the class of pre-Schreier domains~\cite{mZ17} (although never published, the ascent of the IDF property for PSP-domains was also pointed out by the second author at the end of~\cite{mZ17}). In order to establish the ascent of the IDF property for PSP-domains, prime upper to zero ideals will be used as combinatorial tools. In Section~\ref{sec:UTZ ideals and ascending of IDF in PSP and pre-Schreier domains}, we also show how prime upper to zero ideals can be used to characterize  UMT-domains (i.e., integral domains whose prime upper to zero ideals are all maximal $t$-ideals).
\smallskip

Following P. L. Clark~\cite{pC17}, we say that an integral domain~$D$ is a Furstenberg domain if every nonunit of~$D$ is divisible by an irreducible. In Section~\ref{sec:TIDF-domains}, we introduce tightly irreducible-divisor-finite domains, and we study them throughout the rest of the paper. We say that an integral domain is a tightly irreducible-divisor-finite domain (or a TIDF-domain) if it is a Furstenberg IDF-domain. It follows from \cite[Theorem~5.1]{AAZ90} that every FFD is a TIDF-domain. We begin Section~\ref{sec:TIDF-domains} characterizing FFDs in terms of the TIDF property. In addition, we use the TIDF property to characterize UFDs inside the class of AP-domains (i.e., integral domains where every irreducible is prime). We prove in the same section that the TIDF property ascends in the class of PSP-domains and in the class of MCD-finite domains. Finally, we give some examples showing that even if a reduced torsion-free monoid $M$ satisfies the TIDF property, its monoid domain $\qq[M]$ may not be a TIDF-domain.
\smallskip

The $D+M$ construction, which is a useful source of (counter)examples in commutative ring theory, was introduced and first studied by R. Gilmer in~\cite[Appendix II]{rG68} in the context of valuation domains. In addition, the $D+M$ construction in the more general context of integral domains was later investigated by J. Brewer and E.~A. Rutter~\cite{BR76} and also by D. Costa, J. L. Mott, and the second author in~\cite{CMZ78}. In Section~\ref{sec:D+M}, we study the TIDF property through the lens of the $D+M$ construction, establishing some results parallel to those given in \cite[Section~4]{AAZ90} for the IDF property. Following Malcolmson and Okoh~\cite{MO06}, we say that an integral domain~$D$ is a powerful IDF-domain\footnote{We have slightly changed the terminology in~\cite{MO06} to achieve parallelism between the terms `PIDF-domain' and `TIDF-domain'.} (or a PIDF-domain) if for every nonzero $x \in D$ the set $\bigcup_{n \in \nn} \mathsf{D}(x^n)$ is finite up to associates, where $\mathsf{D}(x^n)$ denotes the set of irreducible divisors of $x^n$ in $D$. It is clear that the class of PIDF-domains is a subclass of that consisting of all IDF-domains. In Section~\ref{sec:D+M}, we also investigate how the PIDF property behaves under the $D+M$ construction.
\smallskip

The IDF property is not preserved under localization (see, for instance, \cite[Example~2.3]{AAZ92}). The behavior of the IDF property under localization was first studied by Anderson, Anderson, and the second author in~\cite{AAZ92}, where they found sufficient conditions on a multiplicative subset $S$ of an integral domain~$D$ to ensure that the IDF property transfers from~$D$ to the localization~$D_S$ (see \cite[Theorem~2.1]{AAZ92}). On the other hand, the authors established a Nagata-type result, finding conditions on the multiplicative set~$S$ under which the IDF property transfers from the localization $D_S$ to $D$. In Section~\ref{sec:localization}, we discuss the behavior of both the TIDF and the PIDF properties under localization, expanding to TIDF-domains and PIDF-domains the fundamental results about IDF-domains given in \cite[Theorem~2.1]{AAZ92} and \cite[Theorem~3.1]{AAZ92}.

\bigskip
\section{Background}
\label{sec:background}

\medskip
\subsection{General Notation}

Following standard notation, we let $\zz$, $\qq$, and $\rr$ denote the sets of integers, rational numbers, and real numbers, respectively. In addition, we let $\pp$, $\nn$ and $\nn_0$ denote the sets of primes, positive integers, and nonnegative integers, respectively. 
For $a,b \in \zz$, we let $\ldb a,b \rdb$ denote the discrete interval $\{n \in \zz \mid a \le n \le b\}$, allowing $\ldb a,b \rdb$ to be empty when $a > b$. In addition, given $S \subseteq \rr$ and $r \in \rr$, we set $S_{\ge r} = \{s \in S \mid s \ge r\}$ and $S_{> r} = \{s \in S \mid s > r\}$. For $q \in \qq \setminus \{0\}$, we let $\mathsf{n}(q)$ and $\mathsf{d}(q)$ denote, respectively, the unique $n \in \nn$ and $d \in \zz$ such that $q = n/d$ and $\gcd(n,d) = 1$. Accordingly, for any $Q \subseteq \qq \setminus \{0\}$, we set $\mathsf{n}(Q) = \{\mathsf{n}(q) \mid q \in Q\}$ and $\mathsf{d}(Q) = \{\mathsf{d}(q) \mid q \in Q\}$. Finally, for each $p \in \pp$ and $n \in \zz \setminus \{0\}$, we let $v_p(n)$ denote the maximum $v \in \nn_0$ such that $p^v$ divides $n$, and for $q \in \qq \setminus \{0\}$, we set $v_p(q) = v_p(\mathsf{n}(q)) - v_p(\mathsf{d}(q))$ (in other words, $v_p$ is the $p$-adic valuation map of~$\qq$).

\medskip
\subsection{Monoids}

A \emph{monoid} is a semigroup with an identity element. In the scope of this paper, we tacitly assume that monoids are both cancellative and commutative. The most fundamental role played by monoids in this paper is as multiplicative monoids of integral domains and, accordingly, unless we specify otherwise, we will write monoids multiplicatively. The second relevant role played by monoids here is as exponents of monoid rings, and in this setting we use additive notation as Gilmer does in~\cite{rG84}. Let $M$ be a monoid. 
An element $b \in M$ is \emph{invertible} if there exists $c \in M$ such that $b c = 1$. We let $\mathcal{U}(M)$ denote the group of invertible elements of $M$. When $M$ is the multiplicative monoid of an integral domain $D$, the invertible elements of $M$ are precisely the units of $D$, in which case, we will use the standard notation $D^\times$ instead of $\mathcal{U}(D \setminus \{0\})$. We let $M_{\text{red}}$ denote the quotient monoid $M/\mathcal{U}(M)$. The monoid $M$ is \emph{reduced} if $\mathcal{U}(M)$ is the trivial group, in which case, $M$ is naturally isomorphic to $M_{\text{red}}$. The \emph{quotient group} of $M$ (or \emph{difference group} of $M$ if $M$ is additively written), denoted by $\gp(M)$, is the unique abelian group up to isomorphism satisfying that any abelian group containing a homomorphic image of $M$ will also contain a homomorphic image of $\gp(M)$. The monoid $M$ is \emph{torsion-free} if $\gp(M)$ is a torsion-free group (or equivalently, for all $b,c \in M$, if $b^n = c^n$ for some $n \in \nn$, then $b=c$).
\smallskip

For a subset $S$ of $M$, we let $\langle S \rangle$ denote the submonoid of $M$ generated by $S$, that is, the smallest (under inclusion) submonoid of $M$ containing $S$. An \emph{ideal} of $M$ is a subset $I$ of $M$ such that $I M \subseteq I$ (or, equivalently, $I M = I$). An ideal $I$ of $M$ is \emph{principal} if there exists $b \in M$ satisfying $I = bM$. For $b_1, b_2 \in M$, we say that $b_2$ \emph{divides} $b_1$ in $M$ and write $b_2 \mid_M b_1$ if $b_1 M \subseteq b_2 M$, and we say that~$b_1$ and~$b_2$ are \emph{associates} if $b_1 M = b_2 M$. The monoid $M$ is a \emph{valuation monoid} if for any $b_1, b_2 \in M$ either $b_1 \mid_M b_2$ or $b_2 \mid_M b_1$. A submonoid $N$ of $M$ is \emph{divisor-closed} if for any $b \in N$ and $d \in M$, the relation $d \mid_M b$ implies that $d \in N$. The monoid $M$ satisfies the \emph{ascending chain condition on principal ideals} (\emph{ACCP}) if every increasing sequence (under inclusion) of principal ideals eventually terminates. An element $a \in M \! \setminus \! \uu(M)$ is an \emph{atom} (or an \emph{irreducible}) if whenever $a = u v$ for some $u,v \in M$, then either $u \in \uu(M)$ or $v \in \uu(M)$. We let $\mathcal{A}(M)$ denote the set of atoms of~$M$. The monoid~$M$ is \emph{atomic} if every non-invertible element factors into atoms. One can readily check that every monoid satisfying ACCP is atomic.

\medskip
\subsection{Factorizations}

It is clear that $M$ is atomic if and only if $M_{\text{red}}$ is atomic. Let $\mathsf{Z}(M)$ denote the free (commutative) monoid on $\mathcal{A}(M_{\text{red}})$, and let $\pi \colon \mathsf{Z}(M) \to M_\text{red}$ be the unique monoid homomorphism fixing the set $\mathcal{A}(M_{\text{red}})$. For each $b \in M$, we set
\[	
	\mathsf{Z}(b) = \mathsf{Z}_M(b) = \pi^{-1} (b  \uu(M)),
\]
and call the elements of $\mathsf{Z}(b)$ \emph{factorizations} of~$b$. Observe that $M$ is atomic if and only if $\mathsf{Z}(b)$ is nonempty for any $b \in M$. The monoid $M$ is a \emph{finite factorization monoid} (\emph{FFM}) if it is atomic and $|\mathsf{Z}(b)| < \infty$ for every $b \in M$. In addition, $M$ is a \emph{unique factorization monoid} (\emph{UFM}) if $|\mathsf{Z}(b)| = 1$ for every $b \in M$. By definition, every UFM is an FFM. If $z = a_1   \cdots  a_\ell \in \mathsf{Z}(M)$ for some $a_1, \dots, a_\ell \in \mathcal{A}(M_{\text{red}})$, then $\ell$ is the \emph{length} of the factorization $z$ and is denoted by $|z|$. For each $b \in M$, we set
\[
	\mathsf{L}(b) = \mathsf{L}_M(b) = \{ |z| \mid z \in \mathsf{Z}(b) \}.
\]
The monoid $M$ is a \emph{bounded factorization monoid} (\emph{BFM}) if it is atomic and $|\mathsf{L}(b)| < \infty$ for all $b \in M$. Clearly, every FFM is a BFM. On the other hand, the reader can verify that every BFM satisfies ACCP (see \cite[Corollary~1.3.3]{GH06}).
\smallskip

For the rest of this section, let $D$ be an integral domain. We let $D^*$ denote the multiplicative monoid $D \setminus \{0\}$ of $D$. A subring~$S$ of $D$ is a \emph{divisor-closed subring} of $D$ provided that $S^*$ is a divisor-closed submonoid of $D^*$. Every factorization property defined for monoids in the previous paragraph can be adapted to integral domains in a natural way. The integral domain $D$ is a \emph{unique} (resp., \emph{finite}, \emph{bounded}) \emph{factorization domain} provided that $D^*$ is a unique (resp., finite, bounded) factorization monoid. Accordingly, we use the acronyms UFD, FFD, and BFD. Observe that this new definition of a UFD coincides with the standard definition of a UFD. 
As for monoids, we let $\mathcal{A}(D)$ denote the set of irreducibles of $D$ (for integral domains, the term `irreducible' is somehow more standard than the term `atom'). A generous recompilation of the advances in factorization theory in atomic monoids and domains until 2006 is provided in~\cite{GH06} by A. Geroldinger and F. Halter-Koch.

\medskip
\subsection{Integral Domains and Monoid Domains}

As mentioned in the introduction, the primary purpose of this paper is to provide fundamental results on integral domains where each nonzero element has only finitely many irreducible divisors up to associates. For each $r \in D$, we set
\[
	\mathsf{D}(r) := \{a D^\times \mid a \in \mathcal{A}(D) \text{ and } a \mid_D r\}.
\]
Then we say that $D$ is an \emph{irreducible-divisor-finite domain}, or an \emph{IDF-domain} for short, if $\mathsf{D}(r)$ is finite for every $r \in D^*$. These integral domains were introduced and first investigated by Grams and Warner in~\cite{GW75}. The integral domain~$D$ is a \emph{Furstenberg domain} if every nonunit of $D^*$ is divisible by an irreducible. Now we introduce the following definitions, which are central in this paper.

\begin{definition}
	Let $D$ be an integral domain. 
	\begin{itemize}
		\item We say that~$D$ is a \emph{powerful IDF-domain} (or a \emph{PIDF-domain}) if for every nonzero $r \in D$ the set $\bigcup_{n \in \nn} \mathsf{D}(r^n)$ is finite.
		\smallskip
		
		\item We say that an integral domain is a \emph{tightly IDF-domain} (or a \emph{TIDF-domain}) if it is a Furstenberg IDF-domain.
	\end{itemize}
\end{definition}
\smallskip

We proceed to recall further classes of integral domains generalizing GCD-domains (and so UFDs) that will show up in coming sections. The integral domain $D$ is an \emph{atom-prime domain} (or an \emph{AP-domain}) provided that every irreducible of $D$ is prime. An element $r \in D$ is \emph{primal} if for any $s_1, s_2 \in D$, the relation $r \mid_D s_1 s_2$ guarantees the existence of $r_1, r_2 \in D$ with $r = r_1 r_2$ such that $r_1 \mid_D s_1$ and $r_2 \mid_D s_2$. Following \cite{mZ87}, we say that $D$ is a \emph{pre-Schreier domain} if every $r \in D^*$ is \emph{primal}. Every pre-Schreier domain is an AP-domain. Following \cite{pC68}, we say that an integral domain $D$ is a \emph{Schreier domain} if $D$ is an integrally closed pre-Schreier domain (it is worth noting that the notion of a pre-Schreier domain was motivated by that of a Schreier domain). It is well known that every GCD-domain is a Schreier domain.
\smallskip

Assume now that $M$ is a torsion-free monoid. Following the terminology in~\cite{rG84}, we let $D[M]$ denote the monoid ring of $M$ over $D$, that is, the ring consisting of all polynomial expressions in an indeterminate~$X$ with exponents in $M$ and coefficients in $D$. It follows from \cite[Theorem~8.1]{rG84} that $D[M]$ is an integral domain (as $M$ is torsion-free and, by definition, cancellative). Accordingly, we often call $D[M]$ a \emph{monoid domain}. In addition, it follows from \cite[Theorem~11.1]{rG84} that
\[
	D[M]^\times = \{rX^u \mid r \in D^\times \text{ and } u \in \uu(M)\}.
\]
In light of \cite[Corollary~3.4]{rG84}, we can assume, without loss of generality, that $M$ is a totally ordered monoid, and we do so. Let $f := c_n X^{q_n} + \dots + c_1 X^{q_1}$ be a nonzero element in $D[M]$ for some coefficients $c_1, \dots, c_n \in D^*$ and exponents $q_1, \dots, q_n \in M$ satisfying $q_n > \dots > q_1$. Then $\deg f := \deg_{D[M]} f := q_n$ and $\text{ord} \, f := \text{ord}_{D[M]} \, f := q_1$ are the \emph{degree} and the \emph{order} of $f$, respectively. In addition, the set $\text{supp} \, f := \text{supp}_{D[M]}(f(x)) := \{q_1, \dots, q_n\}$ is the \emph{support} of $f$. Gilmer in~\cite{rG84} provided a fair exposition of the advances in commutative monoid rings until 1984.

\medskip
\subsection{$t$-ideals}

Let $K$ be the quotient field of $D$, and let $F(D)$ denote the set of nonzero fractional ideals of~$D$. For a fractional ideal $A \in F(D)$, let $A^{-1}$ denote the fractional ideal
\[
	(D :_K A) = \{ x \in K \mid xA \subseteq D \}.
\]
The function $A \mapsto A_v := (A^{-1})^{-1}$ on $F(D)$ is the \emph{$v$-operation} on $D$. Associated to the $v$-operation is the \emph{$t$-operation} on $F(D)$: for each fractional ideal $A$ of $D$,
\[
	A \mapsto A_t := \bigcup \big\{ J_v \mid J \in F(D) \text{ is finitely generated and } J \subseteq A \big\}.
\]
For each fractional ideal $A$ of $D$, we see that $A \subseteq A_t \subseteq A_v$. The $v$-operation and the $t$-operation are examples of the so called star operations (see \cite[Sections~32 and~34]{rG72} or \cite[Chapter~1]{jE19}). A fractional ideal $A \in F(D)$ is a \emph{$v$-ideal} (resp., a \emph{$t$-ideal}) if $A = A_v$ (resp., $A = A_t)$. A $t$-ideal maximal among $t$-ideals is a prime ideal called a \emph{maximal $t$-ideal}. If $A$ is a nonzero ideal with $A_t \neq D$, then $A$ is contained in at least one maximal $t$-ideal. A prime ideal that is also a $t$-ideal is a \emph{prime $t$-ideal}. A fractional ideal $A \in F(D)$ is \emph{$v$-invertible} (resp., \emph{$t$-invertible}) if $(A A^{-1})_v = D$ (resp., $(A A^{-1})_t = D)$. A prime $t$-ideal that is also $t$-invertible is a maximal $t$-ideal \cite[Proposition~1.3]{HZ t-inv}. The ideals $A_1, \dots, A_n$ of $D$ are $t$-\emph{comaximal} if $(A_1, \dots, A_n)_t = D$, where $(A_1, \dots, A_n)$ denotes the ideal generated by $A_1, \dots, A_n$, and an ideal $B$ of $D$ is the \emph{$t$-product} of $A_1, \dots, A_n$ if $B = (A_1 \cdots A_n)_t$. The integral domain~$D$ is a \emph{Pr\"ufer $v$-multiplication domain} (or a \emph{PVMD}) if every nonzero finitely generated ideal of~$D$ is $t$-invertible. It is well known that the class of PVMDs contains all GCD-domains, Krull domains, and Pr\"ufer domains.
\smallskip

The class of PSP-domains is a generalization of that consisting of GCD-domains, and it will play an important role in the next section. The \emph{content} $A_f$ of a polynomial $f \in K[X]$ is the fractional ideal of $D$ generated by the coefficients of $f$. Let $f$ be a polynomial in $D[X]$. Then $f$ is \emph{primitive} (\emph{over} $D$) provided that the ideal $A_f$ is not contained in any proper principal ideal of $D$. On the other hand, $f$ is \emph{super-primitive} (\emph{over} $D$) if $(A_f)_v = D$; that is, $A_f^{-1} = D$. By \cite[Theorem~C]{hT72}, every super-primitive polynomial is primitive. The integral domain $D$ is a \emph{PSP-domain} if every primitive polynomial over $D$ is super-primitive. Every pre-Schreier domain is a PSP-domain (this is \cite[Lemma~2.1]{Z WB}, which although stated for Schreier domains is proved only using properties characterizing pre-Schreier domains). On the other hand, it follows from \cite[Theorem~F]{hT72} and \cite[Proposition~4.1]{AS75} that every PSP-domain is an AP-domain.

\bigskip
\section{Ascent on PSP-domains and Prime Upper to Zero Ideals}
\label{sec:UTZ ideals and ascending of IDF in PSP and pre-Schreier domains}

Assume throughout this section that $D$ is an integral domain with quotient field $K$. A prime ideal $P$ of $D[X]$ is called a \emph{prime upper to zero} of $D$ if $P \cap D = (0)$. Observe that a prime ideal $P$ of $D[X]$ is a prime upper to zero ideal if and only if $P = h(X)K[X] \cap D[X]$ for some irreducible/prime polynomial $h \in K[X]$. Since prime upper to zero ideals have height one, they are $t$-ideals.

\medskip
\subsection{Ascent of the IDF Property on PSP-domains}

The primary purpose of this subsection is to show that the IDF property ascends in the class of PSP-domains; that is, if a PSP-domain $D$ is an IDF-domain, then the polynomial ring $D[X]$ is also an IDF-domain.
\smallskip

It follows from \cite[Theorem 1.4]{HZ t-inv} that a prime upper to zero ideal $P$ of $D$ is a maximal $t$-ideal if and only if~$P$ is $t$-invertible, and this happens precisely when~$P$ contains a polynomial $f$ that is super-primitive over~$D$, which means that $(A_f)_v = D$. Based on these observations, it was concluded in~\cite{GMZ} that if $f$ is a super-primitive polynomial over $D$, then $f(X)D[X]$ is a $t$-product of prime uppers to zero ideals. Using the same fact, the authors in~\cite{GMZ} also concluded that $f(X)D[X]$ is a $t$-product of maximal $t$-ideals. An element $a \in D$ is called a \emph{$t$-invertibility element} if every ideal of $D$ containing~$a$ is $t$-invertible. It was proved in \cite[Theorem~1.3]{GMZ} that $a \in D$ is a $t$-invertibility element if and only if $Da$ is a $t$-product of maximal $t$-ideals of $D$. Our next result makes this conclusion somewhat more obvious. Yet, before we state the following lemma, we note that every non-constant polynomial in $D[X]$ belongs to at most finitely many prime upper to zero ideals, some of which may be $t$-invertible.

\begin{lemma} [Upper-to-zero Representation Lemma]\label{Lemma A}
	Let $f \in D[X]$ be a non-constant polynomial and let $P_1, \dots, P_n$ be the only prime uppers to zero ideals containing $f$ that are maximal $t$-ideals. Then the following statements hold.
	\begin{enumerate}
		\item $f(X)D[X] = (AP_{1}^{r_{1}} \cdots P_{n}^{r_{n}})_{t}$ for some ideal $A$ and $r_1, \dots, r_n \in \nn$ such that $A$ and $P_1^{r_1} \cdots P_n ^{r_n}$ are $t$-comaximal.
		\smallskip
		
		\item If $f$ is super-primitive, then $f(X)D[X] = (P_1^{r_1} \cdots P_n^{r_n})_t$.
		\smallskip
		
		\item $f$ has at most finitely many super-primitive divisors.
	\end{enumerate}
\end{lemma}

\begin{proof}
	(1) The proof of this part can be taken from the proof of \cite[Proposition~3.7]{CDZ}.
	\smallskip
	
	(2) Let $A$ be as in part~(1). Assume, by way of contradiction, that $A_t \neq D[X]$. Then $A$ is contained in a maximal $t$-ideal $M$. In this case, $f \in M$ and $M$ is $t$-invertible because $f$ is super-primitive. Now the fact that the only $t$-invertible maximal $t$-ideals containing $f$ are $P_1, \dots, P_n$ implies that $M \cap D \neq (0)$, whence $M = (M \cap D)[X]$ by \cite[Proposition~1.1]{HZ t-inv}. However, $f \in A \subseteq M$ implies that $D \subseteq M$, which generates a contradiction. Consequently, $A_t = D$, and so
	\[
		f(X)D[X] = (AP_1^{r_1} \cdots P_n^{r_n})_t = (A_t P_1^{r_1} \cdots P_n^{r_n})_t =(P_1^{r_1} \cdots P_n^{r_n})_t.
	\]
	
	(3) Let us call an ideal $I$ a $t$-\emph{divisor} of an ideal $A$ if there is an ideal $J$ such that $A=(IJ)_{t}$. Write $f(X) D[X] = (AP_1^{r_1} \cdots P_n^{r_n})_t$ as in part~(1). Then the proper ideals $P_1^{a_1} \cdots P_n^{a_n}$, where $0 \le a_i \le r_i$ for every $i \in \ldb 1,n \rdb$, are $t$-divisors of $f(X) D[X]$, and they only $t$-divide $P_1^{r_1} \cdots P_n^{r_n}$. Indeed, if $A,B,C$ are ideals of $D$ such that $(A,B)_t = D$ and $A_t \supseteq (BC)_t$, then $A_t \supseteq C_t$: this is because $A_t \supseteq (BC)_{t}$ if and only if
	\[
		A_t = (A,BC)_t = (A,AC,BC)_t = (A,(A,B)C)_t = (A,(A,B)_tC)_t = (A,C)_t,
	\]
	which implies that $A_t \supseteq C_t$. Observe that the inclusion $(P_1^{a_1} \cdots P_n^{a_n})_t \supseteq (A P_1^{r_1} \cdots P_n^{r_n})_t$, along with the fact that the ideals~$A$ and $P_1^{a_1} \cdots P_n^{a_n}$ share no maximal $t$-ideals, guarantees that the inclusion $(P_1^{a_1} \cdots P_n^{a_n})_t \supseteq (P_1^{r_1} \cdots P_n^{r_n})_t$ holds. Now the number of proper $t$-divisors of $(P_1^{r_1} \cdots P_n^{r_n})_t$ is less than $\prod_{i=1}^n (r_i + 1)$ and hence finite. Let $h$ be a super-primitive divisor of $f$ in $D[X]$, and write $f(X) = g(X)h(X)$ for some $g \in D[X]$. It follows from part~(2) that $h(X)D[X] = (P_1^{a_1} \cdots P_n^{a_n})_t$ and, therefore, the equality $(P_1^{r_1} \cdots P_n^{r_n})_t = (P_1^{a_1} \cdots P_n^{a_n})_t (g(X))$ holds. After multiplying both sides by $(P_1^{-a_1} \cdots P_n^{-a_n})$ and applying the $t$-operation, we obtain $(P_1^{r_1 -a_1} \cdots P_n^{r_n -a_n})_t = (g(X))$. On the other hand, the fact that $(g(X)h(X))$ is a principal ideal guarantees that $(g(X)h(X)) = (g(X)h(X))_t$. Consequently, $t$-division acts like ordinary division in this case, and so if $n_f$ denotes the number of non-associate super-primitive divisors of~$f$, then $n_f < \prod_{i=1}^{n}(r_i + 1) < \infty$.
\end{proof}

We are in a position to prove that the IDF property ascends to polynomial rings in the class of PSP-domains.

\begin{theorem} \label{Theorem B} \label{thm:IDF ascend for PSP-domains}
	Let $D$ be a PSP-domain. If $D$ is an IDF-domain, then $D[X]$ is also an IDF-domain.
\end{theorem}

\begin{proof}
	Verifying that $D[X]$ is an IDF-domain entails checking that each nonzero polynomial $g \in D[X]$ is divisible by at most finitely many irreducible divisors up to associates. Fix a nonzero $g \in D[X]$.
	
	If~$g$ is constant, then the fact that $D$ is simultaneously an IDF-domain and a divisor-closed subring of $D[X]$ guarantees that $g$ has only finitely many irreducible divisors in $D[X]$ up to associates. Therefore we assume that~$g$ is not a constant. Obviously, each irreducible divisor of $g$ that comes from $D$ is a divisor of each of the coefficients of $g$ in $D$ and so, up to associates, $g$ has only finitely many irreducible divisors coming from $D$. Thus, if every irreducible divisor of $g$ is constant, we are done.
	
	Suppose, therefore, that $a$ is an irreducible divisor of $g$ in $D[X]$ such that $a \notin D$. As $K[X]$ is a UFD, we can write $a = a_1 \dots a_\ell$ for irreducibles/primes $a_1, \dots, a_\ell \in K[X]$ assuming, without loss of generality, that $a_1 \notin D$. Then, after writing $g = a_1 b$ for some $b \in K[X]$, we see that $g \in a_1(X) K[X] \cap D[X]$, which is a prime upper to zero of~$D$. Since each proper $t$-ideal is contained in a maximal $t$-ideal, we can assume that $g$ is contained in a maximal $t$-ideal. 
	Hence we can assume that for $n \in \nn$ the polynomial~$g$ belongs to precisely $n$ prime upper to zero ideals that are maximal $t$-ideals, namely, $P_1, \dots, P_n$. As we have seen in part~(1) of Lemma~\ref{Lemma A}, the equality $g(X)D[X] = (AP_{1}^{r_{1}} \cdots P_{n}^{r_n})_t$ holds, where $(A,P_{1}^{r_{1}} \cdots P_{n}^{r_{n}})_{t}=D[X]$. Now, let $f$ be an irreducible (primitive) polynomial dividing~$g$ in $D[X]$. Then $(A_f)_v = D$ and, therefore,~$f$ is a super-primitive divisor of~$g$. 
	 Thus, in light of part~(3) of Lemma~\ref{Lemma A}, there can only be a finite number of such irreducible polynomials $f$. Hence we conclude that $D[X]$ is also an IDF-domain.

\end{proof}

Since every pre-Schreier domain is a PSP-domain, as a consequence of Theorem~\ref{thm:IDF ascend for PSP-domains} we obtain the following result, which was first established by the second author in~\cite{mZ17}.

\begin{cor}
	Let $D$ be a pre-Schreier domain. If $D$ is an IDF-domain, then $D[X]$ is also an IDF-domain.
\end{cor}

\bigskip
As every PSP-domain is an AP-domain, it is natural to wonder whether the IDF property ascends in the class of AP-domains.

\begin{question}
	Does the IDF property ascend in the class consisting of AP-domains?
\end{question}

\medskip
\subsection{Further Applications of the Upper-to-zero Representation Lemma}

We proceed to provide further applications of the Upper-to-zero Representation Lemma. Recall that an integral domain $D$ is a Pr\"ufer $v$-multiplication domain (PVMD) if every nonzero finitely generated ideal of $D$ is $t$-invertible. We also recall the following result.

\begin{prop} \cite[Proposition~4]{Z poly} \label{Proposition C}
	Let $D$ be an integrally closed domain, and let
	\[
		S = \{ f \in D[X] \mid (A_f)_v = D \}.
	\]
	Then $D$ is a PVMD if and only if $P \cap S \neq \emptyset$ for every prime upper to zero ideal~$P$ of $D[X]$.
\end{prop}

In light of \cite[Theorem 1.4]{HZ t-inv}, one concludes that $D$ is a PVMD if and only if~$D$ is integrally closed and every prime upper to zero ideal in $D[X]$ is a maximal $t$-ideal \cite[Proposition~3.2]{HZ t-inv}. In fact, Proposition~\ref{Proposition C} and \cite[Proposition~2.6]{HMM} led to the notion of a \emph{UMT-domain}, that is, an integral domain whose prime upper to zero ideals are maximal $t$-ideals.

\begin{lemma} \label{Lemma C1}
	Let $B$ be a $t$-invertible $t$-ideal of $D[X]$ with $B \cap D = (0)$. Then $B=(A P_{1}^{r_{1}}P_{2}^{r_{2}} \cdots P_{n}^{r_{n}})_{t}$, where $P_1, \dots, P_n$ are the $t$-invertible prime upper to zero ideals in $D[X]$ containing the ideal $B$ and $A$ is an ideal of $D[X]$ such that $(A, P_{1}^{r_{1}}P_{2}^{r_{2}} \cdots P_{n}^{r_{n}})_{t}=D$.
\end{lemma}

\begin{proof}
	Note that the localization $D[X]_{P_1}$ is a rank-one DVR and, as a result, there exists $r_1 \in \nn$ such that $B \subseteq (P_1^{r_1})_t$ but $B \nsubseteq (P_1^{r_1+1})_t$. Since $(P_1^{r_1})_t$ is $t$-invertible, we can write $B = (B_1 P_1^{r_1})_t$ for some ideal $B_1$. Since $D[X]_{P_j}$ is a rank-one DVR for every $j \in \ldb 1,n \rdb$, proceeding in a similar manner, we obtain that
	\[
		B = (B_1 P_1^{r_1})_t = (B_2P_1^{r_1}P_2^{r_2})_t = \dots = (B_{n}P_{1}^{r_{1}}P_{2}^{r_{2}} \cdots P_{n}^{r_{n}})_{t}.
	\]
	Now set $A := B_n$, 
	and observe that $A \subseteq D[X]$. 
	Finally, it follows from the fact that the ideals $A$ and $(P_1^{r_1} P_2^{r_2} \cdots P_n^{r_n})_t$ share no maximal $t$-ideals that $(A,P_1^{r_1} P_2^{r_2} \cdots P_n^{r_n})_t = D[X]$.
\end{proof}

We can now characterize PVMDs using prime upper to zero ideals.

\begin{prop} \label{Prop D}
	An integral domain $D$ is a PVMD if and only if for each non-constant polynomial $f \in D[X]$, there are prime upper to zero ideals $P_1, \dots, P_n$ such that $f(X)D[X]=(AP_{1}^{r_{1}} \cdots P_{n}^{r_{n}})_{t}$ for some $r_1, \dots, r_n \in \nn$, where $A=A_{f}[X]$.
\end{prop}

\begin{proof}
	The direct implication follows from~\cite[Lemma~2.4]{gwC15}.
	\smallskip

	Conversely, suppose that for each non-constant polynomial $f(X) \in D[X]$ there are prime upper to zero ideals $P_1, \dots, P_n$ such that $f(X)D[X] = (A_{f} P_{1}^{r_{1}}P_{2}^{r_{2}} \cdots P_{n}^{r_{n}})_t$ for some $r_1, \dots, r_n \in \nn$. Then, by construction, $A_f$ is $t$-invertible. Since for every finitely generated nonzero ideal $A=(a_{0},a_{1}, \dots, a_{m})$ we can construct a non-constant polynomial $f = \sum_{i=0}^{m}a_{i}X^{i}$ such that $A_f = A$, we conclude that every finitely generated nonzero ideal of $D$ is $t$-invertible; that is, $D$ is a PVMD.
\end{proof}

We can also use prime upper to zero ideals to characterize PVMDs inside the class consisting of integrally closed domains.

\begin{prop} \label{Proposition E} \label{Proposition F}
	For an integrally closed domain $D$, the following conditions are equivalent.
	\begin{enumerate}
		\item[(a)] $D$ is a PVMD.
		\smallskip
		
		\item[(b)] Every linear non-constant polynomial over $D$ is contained in a $t$-invertible prime upper to zero ideal.
		\smallskip
		
		\item[(c)] Every ideal $A$ of $D[X]$ with $A\cap D=(0)$ is contained in a $t$-invertible prime upper to zero ideal.
	\end{enumerate}
\end{prop}

\begin{proof}
	(a) $\Rightarrow$ (b): If $D$ is a PVMD, then as every prime upper to zero ideal is a maximal $t$-ideal and hence $t$-invertible, every linear polynomial is contained in a $t$-invertible prime upper to zero ideal.
	\smallskip
	
	(b) $\Rightarrow$ (a): Proving this amounts to arguing that each two-generated nonzero ideal of $D$ is $t$-invertible. To do so, take $a,b \in D^*$ such that $(a,b)$ is a nonzero ideal of $D$. Then $f(X) = aX + b$ is a non-constant linear polynomial over $D$, and so $f(X)$ is contained in a $t$-invertible prime upper to zero ideal $P$. Then $f(X)D[X]\subseteq P$, and so $f(X)D[X]=(AP)_t$ for some ideal $A$. As a result,
	\begin{equation} \label{eq:P}
		P=((A^{-1})f(X)D[X])_t = f(X)A^{-1}.
	\end{equation}
	Since $D$ is integrally closed, $P=f(X)K[X] \cap D[X] = f(X)A_f^{-1}[X]$. This, together with~\eqref{eq:P}, guarantees that $ f(X)A^{-1} =  f(X)A_f^{-1}[X]$, from which the equality $A^{-1} = A_f^{-1}[X]$ follows. This, in turns, implies that $A_v = (A_f)_v[X]$. Because $A_v$ and $(A_f)_v$ are of finite type, $A_t = (A_f)_t[X]$. Therefore
	\[
		f(X)D[X] = (AP)_t = (A_t P)_t = ((A_f)_t P)_t = (A_f P)_t,
	\]
	forcing the ideal $(a,b) = A_f$ to be $t$-invertible. Hence we conclude that $D$ is a PVMD.
	\smallskip
	
	(a) $\Rightarrow$ (c): Let $A$ be an ideal of $D[X]$ with $A \cap D=(0)$. 
	Because $D$ is integrally closed, \cite[Theorem 2.1]{AKZ} guarantees the existence of $s \in D^*$, a polynomial $f \in D[X]$, and an ideal $C$ of $D[X]$ with  $C \cap D \neq (0)$ such that $sA = fC$. Since $f(X)D[X]$ is contained in at least one prime upper to zero ideal, $sA$ is also contained in a prime upper to zero ideal. As $s$ is a constant, it does not belong to any prime upper to zero ideal and, therefore, $A$ is contained in at least one prime upper to zero ideal. 
	\smallskip
	
	(c) $\Rightarrow$ (a): Let $f$ be a non-constant linear polynomial. As $D$ is integrally closed, $f(X)A_{f}^{-1}[X] = P$. Since $P$ is $t$-invertible, the ideal $A_{f}^{-1}[X]$ is $t$-invertible and, therefore, so is $A_{f}^{-1}$. Hence the ideal $(A_f)_v$ is also $t$-invertible. Thus, every two-generated nonzero ideal of $D$ is $t$-invertible, which means that $D$ is a PVMD.
\end{proof}

By \cite[Proposition 3.2]{HZ t-inv}, an integral domain $D$ is a PVMD if and only if $D$ is an integrally closed UMT-domain. It turns out that we can also characterize UMT-domains using prime upper to zero ideals.

\begin{prop} \label{Proposition G} \label{prop:UMT characterization}
	An integral domain $D$ is a UMT-domain if and only if each $t$-invertible $t$-ideal $A$ of $D[X]$ with $A\cap D=(0)$ is contained in a $t$-invertible prime upper to zero ideal.
\end{prop}

\begin{proof}
	Assume first that $D$ is a UMT-domain. Since $A \cap D=(0)$, the equality $AK[X] = h(X)K[X]$ holds for some nonzero nonunit $h \in K[X]$. As $h$ is a product of primes in $K[X]$, there is an irreducible element $p \in K[X]$ such that $AK[X]=h(X)K[X] \subseteq p(X)K[X]$. As a result,
	\[
		A \subseteq AK[X] \cap D[X] \subseteq p(X)K[X] \cap D[X] \subseteq D[X].
	\]
	Observe that $p(X)K[X] \cap D[X] \subseteq D[X]$ is a prime upper to zero ideal, and so the fact that $D$ is a UMT-domain ensures that $p(X)K[X] \cap D[X]$ is a maximal $t$-ideal. Hence $A$ is contained in a $t$-invertible prime upper to zero ideal.

	Conversely, assume that each $t$-invertible $t$-ideal $A$ of $D[X]$ with $A\cap D=(0)$ is contained in a $t$-invertible prime upper to zero ideal. 
	Now let $P$ be a prime upper to zero ideal. Then we can write $P = h(X)K[X] \cap D[X]$ for some non-constant polynomial $h \in D[X]$. Note that $h(X)D[X]$ is a $t$-invertible $t$-ideal, and so it must be contained in a $t$-invertible prime upper to zero ideal, namely, $Q$. As a consequence, $P = h(X)K[X] \cap D[X] \subseteq QK[X] \cap D[X] = Q,$ forcing the equality $P=Q$, which means that $P$ is a maximal $t$-ideal.
\end{proof}

\bigskip
\section{Furstenberg IDF-Domains}
\label{sec:TIDF-domains}

In this section, we begin our study of TIDF-domains. Recall that an integral domain $D$ is a TIDF-domain if $D$ is a Furstenberg IDF-domain; that is, every nonunit of $D$ has a \emph{nonempty} finite set of irreducible divisors up to associates. There are, however, Furstenberg domains that are not TIDF-domains.

\begin{example}
	Consider the integral domain $D = \zz + X\qq[X]$. Observe that $\pp \subseteq \mathcal{A}(D)$. Let $f$ be a nonzero nonunit in $D$. Assume first that $\text{ord} \, f = 0$. If there is a $p \in \pp$ such that $p \mid_\zz f(0)$, then $p \mid_D f$. Otherwise, $f(0) \in \{\pm 1\}$ and so every minimum-degree non-constant polynomial dividing $f$ in~$D$ is irreducible. On the other hand, if $\text{ord} \, f \neq 0$, then $2$ (or any other prime) must divide $f$ in $D$. Hence~$D$ is a Furstenberg domain. Finally, distinct primes in $\pp$ are non-associates in $D$ and so the fact that $p \mid_D X$ for all $p \in \pp$ ensures that $D$ is not an IDF-domain, and so $D$ is not a TIDF-domain.
\end{example}

Although it follows from the corresponding definitions that every TIDF-domain is, in particular, an IDF-domain, the converse does not hold in general. Indeed, fields and antimatter domains are IDF-domains but not TIDF-domains. A less trivial example of an IDF-domain that is not a TIDF-domain is given below in Example~\ref{ex:IDF not TIDF}. First, let us describe a construction that allows us to produce TIDF-domains.

\begin{example} \label{ex:TIDF-domains}
	Let $V$ be a discrete rank-$2$ valuation domain with maximal ideal $M$, and let $S$ be the multiplicative set generated by a prime $p$ that belongs to $M$. Now consider the integral domain $D := V +XV_S[X]$. Since $D$ is a Schreier domain (see \cite{CMZ78} and \cite[Example~20]{mZ78}), every irreducible in~$D$ is a prime. Let $f$ be a nonzero nonunit in $D$, and write $f(X) = v+Xg(X)$ for some $v \in V$ and $g \in V_S[X]$. It is clear that $f$ is divisible by $p$ in $D$ when $v$ is a nonunit in $V$ and, hence, in~$D$. On the other hand, when $v$ is a unit in $V$, we can use a simple degree consideration to take a maximum $m \in \nn$ such that $f = a_1 \cdots a_m$ for nonunits $a_1, \dots, a_m$ of $D$, in which case, $a_1$ is an irreducible divisor of~$f$ in~$D$. Thus,~$D$ is a Furstenberg domain. To argue that $D$ is an IDF-domain, suppose that $a$ is an irreducible divisor of $f$ in $D$. If $a$ is a constant polynomial, then $a$ and $p$ are associates in $D$. On the other hand, if $\deg a \ge 1$, then $a$ is associate in $D$ with an element of the form $1 + X g_a(X)$ for some $g_a \in V_S[X]$. Since $D$ is an AP-domain, a degree consideration reveals that only finitely many irreducibles of the form $1 + Xg_a(X)$ can divide $f$ in $D$, whence $f$ has only finitely many divisors in $D$ up to associates. Hence~$D$ is an IDF-domain and so a TIDF-domain.
\end{example}

\begin{example} \label{ex:IDF not TIDF}
	Now suppose that $V$ is a rank-$2$ valuation domain such that the maximal ideal $M$ of~$V$ is idempotent. Take a nonzero $s \in V$ such that $V[s^{-1}]$ is properly contained in the field of fractions of~$V$. Now consider the multiplicative subset $S := \{s^n \mid n \in \nn_0\}$ of $V$, and set $D := V + XV_S[X]$. One can show that $D$ is an IDF-domain by following the argument given in Example~\ref{ex:TIDF-domains}. In addition, as in Example~\ref{ex:TIDF-domains}, the integral domain $D$ is a Schreier domain and, therefore, an AP-domain (see \cite[Example~20]{mZ78}). Because $M$ is idempotent, $V$ cannot contain any primes, which in this case means that $V$ is antimatter. Since $V$ is a divisor-closed subring of $D$, the nonzero elements in~$M$, which are nonunits of $D$, cannot be divisible by any irreducible in $D$. Hence $D$ is not a Furstenberg domain, and so we conclude that $D$ is not a TIDF-domain.
\end{example}

Constructions similar to those given in the previous examples will be considered in Section~\ref{sec:D+M}. Examples~\ref{ex:TIDF-domains} and~\ref{ex:IDF not TIDF}, along with further examples of TIDF-domains, were recently given by the second author in~\cite{mZ22}.

\medskip
\subsection{Atomic Domains}

Since FFDs can be characterized as atomic IDF-domains \cite[Theorem~5.1]{AAZ90}, FFDs are examples of atomic TIDF-domains. However, there are TIDF-domains that are not atomic. The following example sheds some light upon this observation.

\begin{example} \label{ex:nontrivial non-atomic TIDF-domain}
	Fix $p \in \pp$, and then set $D = \zz_{(p)} + X \cc \ldb X \rdb$. Since $\cc\ldb X \rdb$ is local, ~\cite[Lemma~4.17]{AG22} guarantees that $D^\times = \zz_{(p)}^\times + X \cc\ldb  X \rdb$. It is clear that $p \in \mathcal{A}(D)$. Also, if $f \in D^*$ with $\text{ord} \, f \ge 1$, then $f \notin \mathcal{A}(D)$; indeed, $f = p(f/p)$ and both $p$ and $f/p$ belong to $D^* \setminus D^\times$. Therefore, for any $q \in \zz_{(p)}$ and $g \in \cc\ldb X \rdb$, we see that $q + Xg(X)$ belongs to $\mathcal{A}(D)$ if and only if the $p$-adic valuation of $q$ is~$1$. Hence $\mathcal{A}(D) = pD^\times$. This, together with the fact that $p$ divides each nonunit in $D$, implies that $D$ is a TIDF-domain. However, as $\zz_{(p)}$ is not a field, \cite[Proposition~1.2]{AAZ90} ensures that~$D$ is not atomic.
\end{example}

On the other hand, it is natural to wonder under which extra conditions a TIDF-domain is guaranteed to be an FFD. The next proposition gives an answer to this question. Recall that an integral domain~$D$ is Archimedean if $\bigcap_{n \in \nn} x^nD = (0)$ for every nonunit $x \in D$.

\begin{prop} \label{prop:xe}
	The following conditions are equivalent for an integral domain $D$.
	\begin{enumerate}
		\item[(a)] $D$ is an FFD.
		\smallskip
		
		\item[(b)] $D$ is an Archimedean TIDF-domain.
		\smallskip
		
		\item[(c)] $D$ is a TIDF-domain and $\bigcap_{n \in \nn} a^n D = (0)$ for each $a \in \mathcal{A}(D)$.
	\end{enumerate}
\end{prop}

\begin{proof}
	(a) $\Rightarrow $ (b): Since $D$ is an FFD, it is an atomic IDF-domain and, therefore, a TIDF-domain. In addition, $D$ satisfies ACCP because it is an FFD. Thus, it follows from \cite[Theorem~2.1]{BD76} that $D$ is Archimedean.
	\smallskip
	
	(b) $\Rightarrow $ (c): This is obvious. 
	\smallskip
	
	(c) $\Rightarrow $ (a): Since $D$ is a TIDF-domain, it is also an IDF-domain. Thus, all we need to show is that~$D$ is atomic. To do so, let $x_1$ be a nonzero nonunit of~$D$. Since $D$ is a TIDF-domain, we can take $a_1 \in \mathcal{A}(D)$ such that $a_1 \mid_D x_1$. Because $a_1$ is an atom, $\bigcap_{n \in \nn}  a_1^n D = (0).$ So there is an $n_1 \in \nn$ such that $a_1^{n_1} \mid_D x_1$ and $a_1^{n_{1}+1} \nmid_D x_1$. Set $x_2 = x_1/a_{1}^{n_{1}}$. If $x_2$ is a nonunit, then we can proceed as before to obtain $a_2 \in \mathcal{A}(D)$ and $n_2 \in \nn$ such that $a_2^{n_{2}} \mid_D x_2$ but $a_{2}^{n_{2}+1} \nmid_D x_2$. Continuing in this fashion, we can find $a_1, \dots, a_r \in \mathcal{A}(D)$ and $n_1, \dots, n_r \in \nn$ such that $x_{r+1} := x/(a_{1}^{n_{1}} a_{2}^{n_{2}} \cdots a_{r}^{n_{r}}) \in D^*$. Observe that this process cannot continue indefinitely because $x_1$ is divisible in $D$ by only finitely many atoms up to associates. Hence $x_1$ is a product of atoms. Since the choice of $x_1$ was arbitrary, we conclude that~$D$ is atomic. Thus, $D$ is an atomic IDF-domains, and so an FFD.
\end{proof}
\smallskip

Although every integral domain satisfying ACCP is Archimedean, satisfying ACCP does not prevent an integral domain from having a nonzero nonunit with an infinite number of non-associated irreducible divisors. For example, the ring $\qq+X\rr[X]$ satisfies ACCP by \cite[Proposition~1.2]{AAZ90} but its element~$X^2$ has infinitely many distinct irreducible divisors, namely, $X/r$ for any $r \in \rr \setminus \qq$. Thus, the TIDF condition is needed in the statement of Proposition~\ref{prop:xe}. Yet the TIDF condition and atomicity together are excessive as we already know that the IDF condition and atomicity together already guarantee the finite factorization property. On the other hand, the TIDF condition alone does not even guarantee atomicity; for instance, we have already seen in Example~\ref{ex:nontrivial non-atomic TIDF-domain} that $D = \zz_{(p)} + X \cc \ldb X \rdb$ is a non-atomic TIDF-domain. Moreover, we have shown in the same example that $\mathcal{A}(D) = p D^\times$, and we can readily see that $p$ is prime in $D$. Thus, $D$ is an AP-domain. This thwarts any hope of using the TIDF condition in tandem with the AP condition to achieve the finite factorization property. However, in the class consisting of AP-domains, we can refine our understanding of the TIDF property.
\smallskip

\begin{cor} \label{cor:xg}
	For an AP-domain $D$, the following conditions are equivalent.
	\begin{enumerate}
		\item[(a)] $D$ is a UFD.
		\smallskip
		
		\item[(b)]$D$ is a completely integrally closed TIDF-domain.
		\smallskip
		
		\item[(c)] $D$ is an Archimedean TIDF-domain.
		\smallskip
		
		\item[(d)] $D$ is a TIDF-domain such that $\bigcap_{n \in \nn} (p^{n}) = (0)$ for every prime element $p$.
	\end{enumerate}
\end{cor}

\begin{proof}
	(a) $\Rightarrow$ (b) $\Rightarrow$ (c) $\Rightarrow$ (d): These implications follow immediately.
	\smallskip
	
	(d) $\Rightarrow$ (a): As $D$ is an AP-domain by assumption, the condition~(c) of Proposition~\ref{prop:xe} holds and, therefore, $D$ is an FFD. In particular, $D$ is atomic. Since $D$ is an atomic AP-domain, it must be a UFD.
%
%
\end{proof}

\smallskip
\subsection{Polynomial Rings and Monoid Rings}

As proved by Malcolmson and Okoh in~\cite{MO09}, the IDF property does not ascend, in general, from an integral domain to its ring of polynomials. However, it was proved in~\cite[Theorem~2.1]{EK18} that the IDF property ascends to rings of polynomials if we restrict to the subclass of MCD-finite domains. A common divisor $d \in D$ of a nonempty subset $S$ of $D^*$ is called a \emph{maximal common divisor} if $\gcd S/d = 1$. Following~\cite{EK18}, we say that an integral domain~$D$ is \emph{MCD-finite} if every finite subset of $D^*$ has only finitely many maximal common divisors up to associates.

As for the case of the IDF property, the TIDF property ascends from an integral domain to its polynomial ring if we restrict to the class of MCD-finite domains. We will deduce this as a consequence of the following proposition. 

\begin{prop} \label{prop:Furstenberg in polynomial rings}
	Let $D$ be an integral domain, and let $\Gamma$ be a nonempty set. Then the following conditions are equivalent.
	\begin{enumerate}
		\item[(a)] $D$ is a Furstenberg domain.
		\smallskip
		
		\item[(b)] $D[X]$ is a Furstenberg domain.
		\smallskip
		
		\item[(c)] $D[X_\gamma \mid \gamma \in \Gamma]$ is a Furstenberg domain, where $D[X_\gamma \mid \gamma \in \Gamma]$ denotes the polynomial ring over~$D$ in the set of indeterminates $\{X_\gamma \mid \gamma \in \Gamma\}$.
	\end{enumerate}
\end{prop}

\begin{proof}
	(a) $\Rightarrow$ (b): Suppose that $D$ is a Furstenberg domain. Take a nonzero nonunit $f \in D[X]$. Let $g \in D[X]$ be an element of minimum degree among those nonunit polynomials in $D[X]$ dividing $f$. Then the minimality condition on the degree of $g$ guarantees that for any $g_1, g_2 \in D[X]$, the equality $g = g_1 g_2$ implies that either $g_1 \in D$ or $g_2 \in D$. If $g$ is irreducible, then we are done. Assume, therefore, that $g$ is not irreducible. Then we can write $g = h_1 h_2$ for some nonunits $h_1, h_2 \in D[X]$. Assume, without loss of generality, that $h_1$ is a nonunit of $D$. Since $D$ is a Furstenberg domain, there exists $a \in \mathcal{A}(D)$ such that $a \mid_D h_1$. Since $D^*$ is a divisor-closed submonoid of $D[X]^*$, we conclude that $a \in \mathcal{A}(D[X])$ such that $a$ divides $f$ in $D[X]$. Thus, $D[X]$ is a Furstenberg domain.
	\smallskip
	
	(b) $\Rightarrow$ (c): Suppose that $D[X]$ is a Furstenberg domain, and set $T := D[X_\gamma \mid \gamma \in \Gamma]$. An inductive argument, in tandem with the previous implication, immediately shows that any ring of polynomials on finitely many indeterminates over~$D$ is a Furstenberg domain when $D$ is a Furstenberg domain. Now take a nonzero nonunit $f \in T$. Then $f \in S := D[X_1, \dots, X_k]$ for some $X_1, \dots, X_k \in \Gamma$. Since $S$ is Furstenberg, we can take $a \in \mathcal{A}(S)$ such that $a \mid_S f$. Now the fact that $S$ is a divisor-closed subring of~$T$ implies that $a \in \mathcal{A}(T)$ and $a \mid_T f$. Hence $T$ is a Furstenberg domain.
	\smallskip
	
	(c) $\Rightarrow$ (a): This is a consequence of the fact that~$D$ is a divisor-closed subring of the Furstenberg domain $D[X_\gamma \mid \gamma \in \Gamma]$.
\end{proof}

Proposition~\ref{prop:Furstenberg in polynomial rings}, in tandem with Theorem~\ref{thm:IDF ascend for PSP-domains} and \cite[Theorem~2.1]{EK18}, implies that the TIDF property ascends in the class of PSP-domains and in the class of MCD-finite domains, respectively.

\begin{cor} \label{cor:TIDF in PSP and MCD-finite}
	For an integral domain $D$, the following statements hold.
	\begin{enumerate}
		\item If $D$ is a PSP-domain, then $D$ is a TIDF-domain if and only if $D[X]$ is a TIDF-domain.
		\smallskip
		
		\item If $D$ is an MCD-finite domain, then $D$ is a TIDF-domain if and only if $D[X]$ is a TIDF-domain.
	\end{enumerate}
\end{cor}

Rings of polynomials are special cases of monoid domains (where the monoid of exponents are $\nn$ and $\zz$). This may trigger the question of whether the TIDF-property ascends in the context of monoid domains. This is not the case in general, even if we only consider the class of group algebras (i.e., group rings over fields).

\begin{example}\footnote{This example was kindly suggested by Jason Juett.}
	Let $D$ denote the group algebra $\cc[\qq]$ of the group $\qq$ over $\cc$. Since $\qq$ is a group, it is a TIDF-monoid. Let us check that $D$ is actually antimatter. Take a nonunit $f \in D^*$. It is clear that $f$ is not a monomial and, after replacing $f$ by one of its associates, we can assume that $\text{ord} \, f = 0$. Then $f \in \cc[X^{1/n}]$ for some $n \in \nn$. Write $f = g(X^{1/n})$ for some $g \in \cc[X]$. After replacing $n$ by $2n$ if necessary, we can further assume that $\deg g > 1$. Because $f$ is not a monomial and $\cc$ is algebraically closed, we can write $g = g_1 g_2$ for two non-constant polynomials $g_1, g_2 \in \cc[X]$. Then $f = g_1(X^{1/n}) g_2(X^{1/n})$. As none of the factors $g_1(X^{1/n})$ and $g_2(X^{1/n})$ is a monomial, they are nonunits in $D$. Thus, $f$ is not irreducible in $D$. As a result, we conclude that $D$ is antimatter, and so it is not a TIDF-domain.
\end{example}

We can even find a torsion-free reduced monoid $M$ that is Furstenberg (resp., an IDF-monoid) such that the monoid domain $\qq[M]$ is not Furstenberg (resp., an IDF-domain). The following two examples shed some light upon our assertion.

\begin{example} \label{ex:monoid domain that is not Furstenberg}
	By \cite[Proposition~5.1]{CG19}, there exists an atomic additive submonoid $M$ of $\qq_{\ge 0}$ satisfying that $\big\{ \frac{1}{2^n} \mid n \in \nn \big\} \subset M \subset \zz\big[ \frac 12, \frac 13\big]$ (see \cite[Section~5]{CG19} for the explicit construction of $M$). Although~$M$ is a reduced atomic monoid, it does not satisfy ACCP because the ascending chain of principal ideals $\big( \frac 1{2^n} + M \big)_{n \in \nn}$ does not stabilize. Now we can consider the monoid domain $R = \ff_2[M]$, where $\ff_2$ is the field of two elements. Each divisor of the polynomial $X^2 + X + 1$ in $\ff_2[M]$ has the form $\big( X^{2 \frac{1}{2^k}} + X^{\frac{1}{2^k}} + 1\big)^m$ for some $m \in \nn_0$. Indeed, if $X^2 + X + 1 = a(X)b(X)$ in $\ff_2[M]$, then
	\[
		a\big(X^{6^k}\big) b\big(X^{6^k}\big) = \big( X^{2 \cdot 3^k} + X^{3^k} + 1 \big)^{2^k}
	\]
	in the UFD $\ff_2[X]$, and so the fact that $X^{2 \cdot 3^k} + X^{3^k} + 1$ is irreducible in $\ff_2[X]$ (see \cite[Lemma~5.3]{CG19}) guarantees that $a(X) = \big( X^{2 \frac{1}{2^k}} + X^{\frac{1}{2^k}}+ 1 \big)^m$ for some $m \in \nn$. After observing that
	\[
		X^{2 \frac{1}{2^k}} + X^{\frac{1}{2^k}}+ 1 = \big( X^{2 \frac{1}{2^{k+1}}} + X^{\frac{1}{2^{k+1}}} + 1 \big)^2
	\]
	for every $k \in \nn$, one can conclude that $X^2 + X + 1$ is not divisible by any irreducible in $\ff_2[M]$. Hence $\ff_2[M]$ is not a Furstenberg domain.
\end{example}

For any $p \in \pp$, the positive monoid $M_p = \langle 1/p^n \mid n \in \nn_0 \rangle$ is antimatter and, therefore, an IDF-monoid. However, it was proved in \cite[Example~5.5]{fG22} that the monoid domain $\qq[M_p]$ is not an IDF-domain. As being antimatter violates too drastically the Furstenberg condition, here we exhibit an IDF-monoid $M$ with any prescribed size $|\mathcal{A}(M)|$ such that the monoid domain $\qq[M]$ is not an IDF-domain.

\begin{example} \label{ex:IDF in monoid domains}
	Fix $\ell \in \nn \cup \{\infty\}$, and let $(p_n)_{n=1}^\ell$ be a (possibly finite) strictly increasing sequence of primes. Now consider the positive monoid
	\[
		M_\ell := \bigg\langle \frac 1{2^k}, \frac {p_n + 1}{p_n} \ \Big{|} \ k \in \nn \text{ and } n \in \ldb 1, \ell \rdb \bigg\rangle.
	\]
	It is not hard to verify that $\mathcal{A}(M_\ell) = \big\{ \frac{p_n + 1}{p_n} \mid n \in \ldb 1, \ell \rdb \big\}$, and so $|\mathcal{A}(M_\ell)| = \ell$. This implies that $M_\ell$ is not a Furstenberg monoid because $1$ is not divisible by any atom. We claim that $M_\ell$ is an IDF-monoid. Clearly, $M_\ell$ is an IDF-monoid when $\ell \in \nn$. Then we assume that $\ell = \infty$. Fix a nonzero $q \in M_\ell$, and take $N \in \nn$ large enough so that $q \le p_n$ and $p_n \nmid \mathsf{d}(q)$ for any $n \ge N$. Now write $q = m \frac{p_n + 1}{p_n} + r$ for some $m \in \nn_0$ and $r \in M_\ell$ such that $p_n \nmid \mathsf{d}(r)$. Since $p_n \nmid \mathsf{d}(q) \mathsf{d}(r)$, we can infer from $q = m \frac{p_n + 1}{p_n} + r$ that $p_n \mid m$. Therefore $m=0$ as, otherwise, $q \ge \frac{m}{p_n}(p_n + 1) > p_n$. Hence $q$ is not divisible in $M$ by any atom in $\big\{ \frac{p_n + 1}{p_n} \mid n \ge N \big\}$, and so $q$ has only finitely many irreducible divisors. Hence $M_\ell$ is an IDF-monoid.
	
	Let us argue now that $\qq[M_\ell]$ is not an IDF-domain. It suffices to verify that $X-1$ has infinitely many non-associate irreducible divisors in $\qq[M]$. Since
	\[
		X-1 = \big(X^{\frac 1{2^n}} - 1 \big) \prod_{i=1}^n \big( X^{\frac{1}{2^i}} + 1\big)
	\]
	for every $n$, we see that the sequence $\big( X^{\frac 1{2^n}} + 1 \big)_{n \in \nn}$ is a sequence of non-associate divisors of $X-1$ in $\qq[M_\ell]$. We are done once we show that every term in this sequence is irreducible. Let $M$ denote the submonoid $\big\langle \frac 1{2^n} \mid n \in \nn \big\rangle$ of $M_\ell$. Fix $n \in \nn$. If $a(X)$ is a divisor of $X^{\frac 1{2^n}} + 1$ in $\qq[M_\ell]$, then $\deg a(X) \le 1$ and so $a(X) \in \qq[M]$. Now from the fact that $X^{\frac 1{2^n}} + 1 = \Phi_2(X^{1/2^n})$ we can infer that $X^{\frac 1{2^n}} + 1$ is irreducible in $\qq[M]$ (see \cite[Example~5.5]{fG22}). Hence every term in the sequence $\big( X^{\frac 1{2^n}} + 1 \big)_{n \in \nn}$ is irreducible in $\qq[M_\ell]$, and so we conclude that $\qq[M_\ell]$ is not an IDF-domain.
\end{example}

Observe that additive submonoids of $\qq_{\ge 0}$ play a central role in Example~\ref{ex:monoid domain that is not Furstenberg} and Example~\ref{ex:IDF in monoid domains}. These monoids have been recently considered in connection with strongly primariness~\cite{GGT21}, length-factoriality~\cite{CCGS21}, and atomic semirings~\cite{BCG21}.

In the direction of the previous examples, we do not have an answer for the following question.

\begin{question} \label{quest:TIDF in monoid domains}
	Let $M$ be a torsion-free reduced monoid, and let $D$ be an integral domain. If $M$ and~$D$ both satisfy the TIDF property, is necessarily $D[M]$ a TIDF-domain?
\end{question}

With notation as in Question~\ref{quest:TIDF in monoid domains}, further questions of whether certain atomic properties transfer from $M$ and $D$ to $D[M]$ have been investigated in~\cite{GP74,hK98,AJ15,fG22,GL22}.

\bigskip
\section{The $D+M$ Construction}
\label{sec:D+M}

Let $T$ be an integral domain, and let $K$ and~$M$ be a subfield of~$T$ and a nonzero maximal ideal of~$T$, respectively, such that $T = K + M$. For a subdomain $D$ of $K$, set $R = D + M$. In this section, we consider IDF-domains and TIDF-domains through the lens of the $D+M$ construction. We first consider the case when $D$ is not a field.

\begin{theorem} \label{thm:D+M when $D$ is not a field}
	Let $T$ be an integral domain, and let $K$ and $M$ be a subfield of $T$ and a nonzero maximal ideal of $T$, respectively, such that $T = K + M$. For a subring $D$ of $K$, set $R = D + M$. If~$D$ is not a field, then the following statements hold.
	\begin{enumerate}
		\item If $R$ is an IDF-domain, then the subset $\mathcal{A}(R) \cap D$ of $R$ is finite up to associates.
		\smallskip
		
		\item If $R$ is a TIDF-domain, then the subset $\mathcal{A}(R) \cap D$ of $R$ is finite up to associates, and $\mathcal{A}(D)$ is nonempty provided that $D$ is a divisor-closed subring of~$R$.
	\end{enumerate}

	In addition, if $T$ is a local domain, then the following stronger statements hold.
	\begin{enumerate}
		\item[(3)] $R$ is an IDF-domain if and only if $D$ has only finitely many non-associate irreducibles.
		\smallskip
		
		\item[(4)] $R$ is a TIDF-domain if and only if $D$ is a TIDF-domain with a nonempty finite set of non-associate irreducibles.
		\smallskip
	\end{enumerate}
\end{theorem}

\begin{proof}
	(1) Every nonzero $d \in D$ divides any $m \in M$; indeed, $m = d(d^{-1})m$ and  it is clear that $d^{-1} m \in KM \subseteq M$. Therefore, since $M$ is a nonzero ideal, the fact that $R$ is an IDF-domain, immediately implies that the subset $\mathcal{A}(R) \cap D$ of $R$ is finite up to associates.
	\smallskip
	
	(2) Assume that $R$ is a TIDF-domain. The first statement was proved in the previous part. For the second statement, assume that $D$ is a divisor-closed subring of $R$. Since~$D$ is not a field, it must contain a nonzero nonunit $d$, which must remain a nonunit in $R$. As $R$ is a TIDF-domain, we can take $a \in \mathcal{A}(R)$ such that $a \mid_R d$. Because $D$ is a divisor-closed subring of $R$, we see that $a \in \mathcal{A}(D)$. Thus, $\mathcal{A}(D)$ is not empty.
	\smallskip
	
	(3) This is \cite[Proposition~4.3]{AAZ90}.
	\smallskip
	
	(4) For the direct implication, suppose that $R$ is a TIDF-domain. Let us argue first that $D$ is a Furstenberg domain. Take a nonzero nonunit $d \in D$. Since $d$ is a nonunit of $R$, we can take $a_1 + m_1 \in \mathcal{A}(R)$ with $a_1 \in D$ and $m_1 \in M$ such that $a_1 + m_1 \mid_R d$. The fact that $T$ is local ensures that $1 + M \subseteq R^\times$ (see \cite[Lemma~4.17]{AG22}). Thus, $1 + a_1^{-1}m_1$ belongs to $R^\times$, and so $a_1$ and $a_1 + m_1$ are associates. Hence $a_1 \in \mathcal{A}(R)$. Since $D^\times = R^\times \cap D$, it follows that $a_1 \in \mathcal{A}(D)$. Thus, $D$ is a Furstenberg domain. Since $T$ is a local domain, the inclusion $\mathcal{A}(D) \subseteq \mathcal{A}(R)$ holds, and so it follows from part~(2) that $D$ has only finitely many irreducibles up to associates in~$R$ and, therefore, up to associates in $D$. Hence $D$ is a TIDF-domain and its subset $\mathcal{A}(D)$ is finite up to associates.
	
	For the reverse implication, suppose that $D$ is a TIDF-domain and that the subset $\mathcal{A}(D)$ of $D$ is (nonempty and) finite up to associates. Because of part~(3), it is enough to verify that every nonzero nonunit $r \in R$ has an irreducible divisor in $R$. This is clear if $r \in M$ because $\mathcal{A}(D) \subseteq \mathcal{A}(R)$; we have already observed that every nonzero element of $D$ divides every element of $M$ in $R$. Assume, therefore, that $r \notin M$ and write $r = d + m$ for some $d \in D^*$ and $m \in M$. Since $T$ is a local domain, $1 + d^{-1}m \in R^\times$, and so $r$ and $d$ are associates in $R$. Now the fact that $D$ is a TIDF-domain guarantees that $r$ is divisible by an irreducible in $D$ and, therefore, by an irreducible in~$R$. Hence $R$ is also a TIDF-domain.
\end{proof}

\begin{cor}
	Let $D$ be an integral domain with quotient field $K$, and let $L$ be a field extension of~$K$. Consider the subrings $R = D + X L[X]$ and $S = D + X L \ldb X \rdb$ of $L[X]$ and $L\ldb X \rdb$, respectively. If $D$ is not a field, then the following statements hold.
	\begin{enumerate}
		\item $R$ is a TIDF-domain if and only if $\mathcal{A}(D)$ is a nonempty finite subset of $D$ up to associates.
		\smallskip
		
		\item $S$ is a TIDF-domain if and only if $\mathcal{A}(D)$ is a nonempty finite subset of $D$ up to associates.
	\end{enumerate}
\end{cor}

\begin{proof}
	(1) The direct implication follows from Theorem~\ref{thm:D+M when $D$ is not a field} because $D$ is a divisor-closed subring of~$R$. For the reverse implication, consider a general nonzero nonunit element $\alpha X^m(1 + Xf(X))$ of $D + XL[X]$, where $\alpha \in L$ and $f(X) \in L[X]$. Observe that every irreducible divisor of $1 + Xf(X)$ in $R$ is also an irreducible divisor of $1 + Xf(X)$ in $L[X]$ and any two of such divisors are associates in $R$ if and only if they are associates in $L[X]$. As $L[X]$ is a UFD, $1 + Xf(X)$ has a finite number $n_f$ of irreducible divisors in $R$ up to associates. Let $n_D$ be the number of irreducibles of~$D$ up to associates. If $m > 0$, then the number of irreducible divisors of $\alpha X^m(1+Xf(X))$ is $n_D + n_f$. On the other hand, if $m=0$, then $\alpha \in D$, and so the number of irreducible divisors of $\alpha$ in $D$ up to associates is at most $n_D$, in which case, the number of irreducible divisors of $\alpha X^m (1+Xf(X))$ up to associates belongs to $\ldb 1, n_D +n_f \rdb$. Thus, $R$ is a TIDF-domain.
	\smallskip
	
	(2) This is similar to part~(1).
\end{proof}

\medskip
It follows from \cite[Lemma~4.18]{AG22} that $\mathcal{A}(R) \subseteq T^\times \cup \mathcal{A}(T)$ and, if $D$ is a field, $\mathcal{A}(R) \subseteq \mathcal{A}(T)$. We proceed to consider the TIDF property under the $D+M$ construction when $D$ is a field.

\begin{theorem} \label{thm:D+M construction when D is a field}
	Let $T$ be an integral domain, and let $K$ and $M$ be a subfield of $T$ and a nonzero maximal ideal of $T$, respectively, such that $T = K + M$. For a subfield $F$ of~$K$, set $R = F + M$.
	\begin{itemize}
		\item When $M \cap \mathcal{A}(R) \neq \emptyset$, then the following statements hold.
		\smallskip
		\begin{enumerate}
			\item[(1)] $R$ is an IDF-domain if and only if $T$ is an IDF-domain and $|K^\times/F^\times| < \infty$.
			\smallskip
			
			\item[(2)] $R$ is a TIDF-domain if and only if $T$ is a TIDF-domain and $|K^\times/F^\times| < \infty$.
		\end{enumerate}
		\smallskip
		
		\item When $M \cap \mathcal{A}(R) = \emptyset$, then the following statements hold.
		\smallskip
		\begin{enumerate}
			\item[(3)] $R$ is an IDF-domain if and only if $T$ is an IDF-domain.
			\smallskip
			
			\item[(4)] $R$ is a TIDF-domain if and only if $T$ is a TIDF-domain.
		\end{enumerate}
	\end{itemize}
\end{theorem}

\begin{proof}
	(1) This is \cite[Proposition~4.2(a)]{AAZ90}.
	\smallskip
	
	(2) For the direct implication, suppose that $R$ is a TIDF-domain. In light of part~(1), we only need to argue that $T$ is a Furstenberg domain. To do so, fix a nonzero nonunit $t \in T$. After replacing $t$ by one of its associates in $T$, we can assume that $t \in R$. Since $t$ is a nonunit in~$T$, it must be a nonunit in~$R$. This, along with the fact that $R$ is a TIDF-domain, ensures the existence of $a \in \mathcal{A}(R)$ such that $a \mid_R t$. Because $F$ is a field, it follows from \cite[Lemma~4.18]{AG22} that $\mathcal{A}(R) \subseteq \mathcal{A}(T)$. Thus, $a \in \mathcal{A}(T)$ such that $a \mid_T t$. Hence $T$ is Furstenberg and so a TIDF-domain.
	\smallskip
	
	Conversely, suppose that $T$ is a TIDF-domain and $|K^\times/F^\times| < \infty$. Because of part~(1), it suffices to show that $R$ is a Furstenberg domain. To do this, fix a nonzero nonunit $r \in R$. From the fact that $F$ is a field, one can deduce that $R^\times = T^\times \cap R$ (see \cite[Lemma~4.7]{AG22}) and, therefore, conclude that $r$ is a nonunit in $T$. We split the rest of the proof in the following two cases.
	\smallskip
	
	\noindent \emph{Case 1:} $r \notin M$. After replacing~$r$ by one of its associates in $R$, we can assume that $r = 1+m$ for some nonzero $m \in M$. Because $1+m \notin T^\times$ and $T$ is a TIDF-domain, $1+m$ has an irreducible divisor in~$T$, and we can assume that such an irreducible divisor of $1+m$ has the form $1 +m'$ for some $m' \in M$ (note that no element of $M$ can divide $1+m$ in~$T$). Now the fact that $1+m' \in \mathcal{A}(T)$ implies that $1+m' \in \mathcal{A}(R)$. Therefore $1+m'$ is an irreducible in $R$ such that $1+m' \mid_R r$. Thus, $R$ is Furstenberg and so a TIDF-domain.
	\smallskip
	
	\noindent \emph{Case 2:} $r \in M$. As $T$ is a TIDF-domain, we can pick $a \in \mathcal{A}(T)$ such that $a \mid_T r$. If $a \notin M$, then $a$ is associate in $T$ with an irreducible element of the form $1 + m$ for some nonzero $m \in M$. Thus, $1 + m \in \mathcal{A}(R)$ and $1+m \mid_R r$. Suppose, on the other hand, that $a \in M$. Since every element in $T$ is associate in $T$ with an element of $R$, after replacing~$a$ by one of its associates in $T$, we can assume that $a \mid_R r$ (note that every associate of $a$ in $T$ belongs to~$M$). As $R^\times = T^\times \cap R$, it follows that $a \in \mathcal{A}(R)$. Hence $R$ is a TIDF-domain.
	\smallskip
	
	(3) This is \cite[Proposition~4.2(b)]{AAZ90}.
	\smallskip
	
	(4) This mimics the argument given for part~(2) as the hypothesis $M \cap \mathcal{A}(R) \neq \emptyset$ was only needed to transfer the IDF property.
\end{proof}

As in the case when $D$ is not a field, we highlight the following special cases.

\begin{cor}
	For a field extension $K \subseteq L$, consider the subrings $R = K + X L[X]$ and $S = K + X L \ldb X \rdb$ of $L[X]$ and $L\ldb X \rdb$, respectively. Then the following statements hold.
	\begin{enumerate}
		\item $R$ is a TIDF-domain if and only if $|L^{\ast}/K^{\ast }|<\infty .$
		\smallskip
		
		\item $S$ is a TIDF-domain if and only if $|L^{\ast}/K^{\ast }| < \infty$.
	\end{enumerate}
\end{cor}

\medskip
We end this section by briefly discussing the $D+M$ construction on the class of PIDF-domains, which is
another natural class of domains strictly contained in that of IDF-domains. Let $R$ be an integral domain. Recall that~$R$ is a PIDF-domain if for every nonzero $x \in R$ the set $\bigcup_{n \in \nn} \mathsf{D}(x^n)$ is finite, where $\mathsf{D}(x)$ denotes the set of associate classes of irreducible divisors of $x$ in~$R$. As it is the case for the IDF property, the PIDF property does not ascend from an integral domain~$D$ to its ring of polynomials (see \cite[Theorem~2.2]{MO09}). Examples of PIDF-domains include Krull domains and, in particular, Dedekind domains and rings of integers of algebraic number fields (see \cite[Corollary~3.3]{MO06}). Every antimatter domain (for instance, the monoid ring $\ff_2[\qq_{\ge 0}]$) is a PIDF-domain that is not a TIDF-domain. As the following example illustrates, there are TIDF-domains that are not PIDF-domains.

\begin{example}
	Fix $p \in \pp$ with $p \equiv 1 \pmod{4}$. Now consider the subring $\zz[i p]$ of the ring of Gaussian integers $\zz[i]$. Since $\zz[i]$ is an FFD and the quotient group $\zz[i]^\times / \zz[i p]^\times$ is finite, it follows from \cite[Theorem~3]{AM96} that $\zz[i p]$ is also an FFD, and so a TIDF-domain. On the other hand, it follows from \cite[Theorem~2.17]{MO06} that $\zz[i p]$ is not a PIDF-domain (the condition $p \equiv 1 \pmod{4}$ is needed for this last statement to hold).
\end{example}

\begin{prop} \label{prop:D+M construction}
	Let $T$ be an integral domain, and let $K$ and $M$ be a subfield of $T$ and a nonzero maximal ideal of $T$, respectively, such that $T = K + M$. For a subfield $F$ of~$K$, set $R = F + M$. Then the following statements hold.
	\begin{enumerate}
		\item If $M \cap \mathcal{A}(R) \neq \emptyset$, then $R$ is a PIDF-domain if and only if $T$ is a PIDF-domain and the group $K^\times/F^\times$ is finite.
		\smallskip
		
		\item If $M \cap \mathcal{A}(R) = \emptyset$, then $R$ is a PIDF-domain if and only if $T$ is a PIDF-domain.
	\end{enumerate}
\end{prop}

\begin{proof}
	(1) Assume that $R$ is a PIDF-domain. In particular, $R$ is an IDF-domain, and it follows from part~(1) of Theorem~\ref{thm:D+M construction when D is a field} that $T$ is an IDF-domain and the group $K^\times/F^\times$ is finite. Suppose, towards a contradiction, that $T$ is not a PIDF-domain. As $T$ is an IDF-domain, for some nonzero $t \in T$ we can choose a strictly increasing sequence $(k_n)_{n \in \nn}$ of positive integers and a sequence $(a_n)_{n \in \nn}$ of pairwise non-associate irreducibles of~$T$ with $a_n \mid_T t^{k_n}$ for every $n \in \nn$. After replacing $t$ by one of its associates in $T$, we can assume that $t = 1 + m$ or $t = m$ for some $m \in M$. Observe now that, for each $n \in \nn$, we can replace~$a_n$ by one of its associates in $T$ so that $a_n = 1 + m_n$ or $a_n = m_n$ for some $m_n \in M$ and $a_n \mid_R t^{k_n}$. Therefore the set $A := \{a_n \mid n \in \nn\}$ is a subset of $\mathcal{A}(R)$ that is not finite up to associates. However, the fact that $aR^\times \in \bigcup_{n \in \nn} \mathsf{D}_R(t^n)$ for all $a \in A$ contradicts that $R$ is a PIDF-domain.
	\smallskip
	
	Conversely, suppose that $T$ is a PIDF-domain and the group $K^\times/F^\times$ is finite, and then take $\alpha_1, \dots, \alpha_m \in K^\times$ such that $K^\times/F^\times = \{\alpha_1 F^\times, \dots, \alpha_m F^\times\}$. Fix a nonzero nonunit $r \in R$. Since~$F$ is a field, it follows from \cite[Lemma~4.18]{AG22} that $\mathcal{A}(R) \subseteq \mathcal{A}(T)$, and so $aT^\times \in \bigcup_{n \in \nn} \mathsf{D}_T(r^n)$ for every $a \in \mathcal{A}(R)$ with $aR^\times \in \bigcup_{n \in \nn} \mathsf{D}_R(r^n)$. Hence we can define $\varphi \colon \bigcup_{n \in \nn} \mathsf{D}_R(r^n) \to \bigcup_{n \in \nn} \mathsf{D}_T(r^n)$ by $\varphi(aR^\times) = aT^\times$. Now for $bT^\times \in \bigcup_{n \in \nn} \mathsf{D}_T(r^n)$, we see that $\varphi^{-1}(bT^\times) \subseteq \{\alpha_1 b R^\times, \dots, \alpha_m b R^\times\}$. This, along with the fact that $\bigcup_{n \in \nn} \mathsf{D}_T(r^n)$ is finite, guarantees that $\bigcup_{n \in \nn} \mathsf{D}_R(r^n)$ is a finite set. As a result,~$R$ is a PIDF-domain.
	\smallskip
	
	(2) This follows the lines of part~(1) as the hypothesis $M \cap \mathcal{A}(R) \neq \emptyset$ was only needed to transfer the IDF property.
\end{proof}

\bigskip
\section{Localization}
\label{sec:localization}

We conclude this paper studying both the TIDF and the PIDF properties through the lens of localization, establishing results parallel to those given in \cite{AAZ92} for the IDF property.
\smallskip

It is well known that the property of being atomic is not preserved under localization. As the following example illustrates, the same is true for the TIDF property.

\begin{example}
	Fix $q \in \qq_{> 1} \setminus \nn$, and consider the monoid $M = \langle q^n \mid n \in \nn_0 \rangle$. One can verify that $M$ is an atomic monoid with $\mathcal{A}(M) = \{q^n \mid n \in \nn_0\}$ (this also follows from \cite[Theorem~4.1]{CG22}). Indeed, it follows from \cite[Proposition~5.6]{fG19} that $M$ is an FFM. Since $q \notin \nn$, the Grothendieck group of $M$ is
	\[
		G = \bigcup_{m \in \nn} \zz q^m = \Big\{ \frac n{d^k} \ \Big{|} \ n \in \zz \text{ and } k \in \nn_0 \Big\},
	\]
	where $d = \mathsf{d}(q)$. Let $p$ be a prime divisor of $d$, and consider the monoid ring $\ff_p[M]$, where $\ff_p$ is the field of $p$ elements. Since the sequence $(q^n)_{n \in \nn}$ of atoms of $M$ increases to infinity, it follows from \cite[Proposition~4.9]{fG22} that $\ff_p[M]$ is an FFD and, therefore, a TIDF-domain. Observe that the localization of $\ff_p[M]$ at the multiplicative set $S = \{X^m \mid m \in M\}$ is the group ring $\ff_p[G]$. Since~$\ff_p$ is a perfect field of characteristic $p$, every nonunit of $\ff_p[G]$ is the $p$-th power of a nonunit, and so the integral domain $\ff_p[G]$ is an antimatter domain. Thus, it is not a TIDF-domain.
\end{example}

However, we will see that under some special conditions on the multiplicative set, the TIDF property is preserved under localization. Let $A \subseteq B$ be a ring extension. Following Cohn~\cite{pC68}, we call $A \subseteq B$ an \emph{inert extension} if $xy \in A$ for $x,y \in B^\ast$ implies that $ux, u^{-1}y \in A$ for some $u \in B^\times$. For an inert extension $A \subseteq B$ of integral domains, one can readily verify that $\ii(A) \subseteq B^\times \cup \ii(B)$. Therefore if $A \subseteq B$ is inert and $A^\times = B^\times \cap A$, then $\ii(A) = \ii(B) \cap A$.

\begin{example} \hfill
	\begin{enumerate}
		\item 	If $R$ is an integral domain, then the extension $R \subseteq R[X]$ is inert.
		\smallskip
		
		\item  Furthermore, under the usual notation of the $D+M$ construction, where $T = K + M$ and $R = D + M$, it is not hard to argue that both extensions $D \subseteq R$ and $R \subseteq T$ are inert.
		\smallskip
		
		\item Fix $n \in \zz$ with $n \ge 2$, and then consider the extension $R[X^n] \subseteq R[X]$. It is clear that $R[X^n]^\times = R^\times$. Observe, on the other hand, that although $X^n \in R[X^n]$ there is no $u \in R^\times$ such that $uX \in R[X^n]$. As a result, $R[X^n] \subseteq R[X]$ is not an inert extension.
	\end{enumerate}
\end{example}

Let $D$ be an integral domain, and let $S$ be a multiplicative set of $D$. If $S$ is generated by primes, then the localization extension $D \subseteq D_S$ is inert by \cite[Proposition~1.9]{AAZ92}. Now if $D \subseteq D_S$ is an inert extension, then we know by \cite[Theorem~2.1]{AAZ92} that $D_S$ is atomic provided that $D$ is atomic. Unfortunately, the same statement is no longer true if atomicity is replaced by the IDF property \cite[Example~2.3]{AAZ92} or the TIDF property as we will illustrate now. The ring in the following example is basically the same ring used in~\cite[Example~2.3]{AAZ92}.

\begin{example} \label{ex:TIDF does not trasfer under inert localization}
	Fix $p \in \pp$. We have seen in Example~\ref{ex:nontrivial non-atomic TIDF-domain} that $D = \zz_{(p)} + X \cc \ldb X \rdb$ is a TIDF-domain (since $\zz_{(p)}$ is a TIDF-domain with a nonempty finite set of irreducibles up to associates, it also follows from part~(4) of Theorem~\ref{thm:D+M when $D$ is not a field} that $D$ is a TIDF-domain). One can easily check that $p$ is also prime in~$D$. Thus, if $S$ is the multiplicative set $\{u p^n \mid u \in D^\times \text{ and } n \in \nn_0\}$, then the extension $D \subseteq D_S$ is inert. However, as argued in \cite[Example~2.3]{AAZ92}, the ring $D_S$ is not even an IDF-domain.
\end{example}

A saturated multiplicative subset $S$ of an integral domain $D$ is called \emph{splitting} if each $r \in D$ can be written as $r = as$ for some $a \in D$ and $s \in S$ such that $aD \cap s'D = a s'D$ for all $s' \in S$. If $S$ is a splitting multiplicative set, then the extension $D \subseteq D_S$ is inert by \cite[Proposition~1.5]{AAZ92}. In addition, if $D$ is atomic, then every multiplicative subset of $D$ generated by primes is a splitting multiplicative set. In general, localization does not preserve irreducibles in integral domains as, after all, fields do not contain irreducibles. However, for an integral domain $D$ and a splitting multiplicative subset $S$ of $D$, if $r = as$ with $a \in D$ and $s \in S$, then $a \in \mathcal{A}(D)$ if and only if $r \in \mathcal{A}(D_S)$ \cite[Corollary~1.4(c)]{AAZ92}.

We proceed to establish the statements of both \cite[Theroem~3.1]{AAZ92} and \cite[Theorem~2.4(a)]{AAZ92} for the case of TIDF-domains.

\begin{prop} \label{prop:TIDF under localization}
	Let $D$ be an integral domain, and let $S$ be a splitting multiplicative subset of $D$ generated by primes. Then $D$ is a TIDF-domain if and only if $D_S$ is a TIDF-domain.
\end{prop}

\begin{proof}
	Let $B$ be the subset of $D$ consisting of all elements that are not divisible in $D$ by any prime contained in $S$. Fix $r \in D$, and write $r = bs$ for some $b \in D$ and $s \in S$ such that $b D \cap s'D = bs' D$ for all $s' \in S$. In particular, if $p$ is a prime in $S$, then $b D \cap pD = bp D$ and, therefore, $p \nmid_D b$. Thus, $b \in B$. Therefore every element $r \in D$ can be written as $r = bs$ for some $b \in B$ and $s \in S$.
%
	\smallskip
	
	By virtue of \cite[Theroem~3.1]{AAZ92} and \cite[Theorem~2.4(a)]{AAZ92}, the ring $D$ is an IDF-domain if and only if $D_S$ is an IDF-domain. Therefore it suffices to show that $D$ is a Furstenberg domain if and only if $D_S$ is a Furstenberg domain. Assume first that $D$ is a Furstenberg domain. Let $r$ be a nonzero nonunit in $D_S$. After replacing $r$ by one of its associates in $D_S$, we can assume that $r \in D$. By the conclusion in the first paragraph, we can write $r = bs$ for some $b \in B$ and $s \in S$. Observe that $b$ is a nonunit in~$D$. Since $D$ is a Furstenberg domain, there is an $a \in \mathcal{A}(D)$ such that $b = a r'$ for some $r' \in D$. Since none of the primes contained in $S$ can divide $a$ in $D$, it follows that $a \notin S = D_S^\times$. Since $D \subseteq D_S$ is an inert extension, $\mathcal{A}(D) \subseteq D_S^\times \cup \mathcal{A}(D_S)$ by \cite[Lemma~1.1]{AAZ92}, and so $\mathcal{A}(D) \setminus D_S^\times \subseteq \mathcal{A}(D_S)$. Hence $a$ is an irreducible divisor of $r$ in $D_S$. As a result, $D_S$ is Furstenberg.
	\smallskip
	
	Conversely, suppose that $D_S$ is a Furstenberg domain. Take $r$ to be a nonzero nonunit in~$D$. We assume that $p \nmid_D r$ for any $p \in S$ as otherwise we are done. Then $r$ is a nonunit in $D_S$. Since $D_S$ is Furstenberg, there exists $a' \in \mathcal{A}(D_S)$ dividing $r$ in $D_S$. After replacing $a'$ by one of its associates in $D_S$, we can assume that $a' \in D$. As $S$ is splitting, $a' = as$ for some $a \in D$ and $s \in S$ such that $a D \cap s' D = as' D$ for all $s' \in S$. Because $S$ is splitting, it follows from \cite[Corollary~1.4]{AAZ92} that $a \in \mathcal{A}(D)$. Write $r = a \frac{b}{t}$ for some $b \in D$ and $t \in S$. Now the fact that $t \mid_D ab$ implies that $t \mid_D b$. As a result, $a$ is an irreducible divisor of $r$ in $D$. Thus, $D$ is a Furstenberg domain.
\end{proof}

As the following example indicates, both conditions, being splitting and being generated by primes, are needed for Proposition~\ref{prop:TIDF under localization} to hold.

\begin{example} \hfill
	\begin{enumerate}
		\item For a prime $p$, we have seen in Example~\ref{ex:TIDF does not trasfer under inert localization} that $D = \zz_{(p)} + X \cc\ldb X \rdb$ is a TIDF-domain but that its localization at the multiplicative set $S := \{u p^n \mid u \in D^\times \text{ and } n \in \nn_0 \}$ is not a TIDF-domain. Observe that $S$ is not a splitting multiplicative subset of $D$ because every power of $p$ divides~$X$ in $D$ (see \cite[Proposition~1.6]{AAZ92}).
		\smallskip
		
		\item Fix $p \in \pp$, and consider the monoid domain $R := \ff_p[\qq_{\ge 0}]$. Since the monoid $\qq_{\ge 0}$ is a GCD-monoid, $R$ is a GCD-domain by \cite[Theorem~6.4]{GP74}. As $\ff_p$ is a perfect field of characteristic~$p$ and $\qq_{\ge 0}$ is $p$-divisible, $R$ is antimatter, and so $R$ is not a UFD. Hence $D := R[Y]$ is a GCD-domain but not a UFD. Now consider the multiplicative subset $S := R^*$ of $D$. Since $R$ is a GCD-domain,~$S$ is splitting. Observe now that $D_S = F[Y]$, where $F$ is the quotient field of~$R$. Then $D_S$ is a UFD and so a TIDF-domain. However, since $R$ is antimatter, it is not a TIDF-domain, and so $D$ is not a TIDF-domain by Corollary~\ref{cor:TIDF in PSP and MCD-finite}. Finally, observe that $S$ is not generated by primes because $R$ is not a UFD.
	\end{enumerate}
\end{example}
\smallskip

We conclude this subsection with a version of Proposition~\ref{prop:TIDF under localization} for PIDF-domains. The proof we give resemblances that of \cite[Theorem~2.4(a)]{AAZ92}.

\begin{prop} \label{prop:PIDF under localization}
	Let $D$ be an integral domain, and let $S$ be a splitting multiplicative subset of $D$ generated by primes. Then $D$ is a PIDF-domain if and only if $D_S$ is a PIDF-domain.
\end{prop}

\begin{proof}
	For the direct implication, fix a nonzero nonunit $x \in D_S$, and let us argue that only finitely many irreducible elements in $D_S$ (up to associates) divide some power of $x$. After replacing $x$ by some of its associates, we can assume that $x \in D$. Take an irreducible divisor $a$ of $x^k$ in $D_S$ for some $k \in \nn$. After replacing $a$ by some of its associates in $D_S$, we can assume that $a \in D$. Furthermore, as $S$ is a splitting multiplicative set, after replacing~$a$ by some element in $aS^{-1}$, we can assume that $aD \cap sD = as D$ for all $s \in S$. The fact that $S$ is splitting also guarantees that $a$ is irreducible in $D$ \cite[Corollary~1.4]{AAZ92}. Now write $x^k = a(d/s)$ for some $d \in D$ and $s \in S$. Then $sx^k = ad \in aD \cap sD = as D$, and so we can take $d' \in D$ such that $sx^k = asd'$; that is, $x^k = ad'$. Thus, $a$ divides $x^k$ in $D$. Hence the fact that $D$ is a PIDF-domain implies that only finitely many irreducibles in $D_S$ (up to associates) divide some power of $x$ in $D_S$. As a result, $D_S$ is a PIDF-domain.
	\smallskip
	
	Conversely, suppose that $D_S$ is a PIDF-domain, and let $x$ be a nonzero nonunit in $D$. Take $y \in D$ to be an irreducible divisor of some power of $x$ in $D$. Now write $x = x' s$ for some $x' \in D$ and $s \in S$ such that no prime in~$S$ divides $x'$ in~$D$. As $D_S$ is a PIDF-domain, only finitely many irreducibles of $D_S$ (up to associates) divide some power of $x'$, namely, $a_1, \dots, a_k$. After replacing $a_1, \dots, a_k$ by suitable associates, we can assume them to be in $D$. Then $a_1, \dots, a_k \in \mathcal{A}(D)$ because $S$ is a splitting multiplicative set of~$D$ \cite[Corollary~1.4]{AAZ92}. If $y \in S$, then $y$ is associate with one of the finitely many primes in $S$ dividing~$s$. Now suppose that $y \notin S$. As $D \subseteq D_S$ is an inert extension, $\mathcal{A}(D) \subseteq D_S^\times \cup \mathcal{A}(D_S)$ and, therefore, $y \in \mathcal{A}(D_S)$. Since $y$ divides some power of $x's$ in $D$, we see that $y$ divides some power of $x'$ in $D$ and so in $D_S$. Then we can take $j \in \ldb 1,k \rdb$ such that $y = a_j (s_1/s_2)$ for some $s_1, s_2 \in S$. Because $S$ is generated by primes, the fact that $a_j, y \in  \mathcal{A}(D) \setminus S$ guarantees that $s_1/s_2 \in D^\times$, and so $y$ and $a_j$ are associates in~$D$. Hence any irreducible divisor of a power of $x$ in $D$ must be associate with either a prime dividing~$s$ or an irreducible in $\{a_1, \dots, a_k\}$. Hence $D$ is a PIDF-domain.
\end{proof}

\bigskip
\section*{Acknowledgments}

The authors are thankful to an anonymous referee for several helpful suggestions that helped to improve an earlier version of this paper. While working on this paper, the first author was kindly supported by the NSF award DMS-1903069.

\bigskip


\begin{thebibliography}{20}

	\bibitem{AAZ90} D.~D. Anderson, D.~F. Anderson, and M. Zafrullah: \emph{Factorizations in integral domains}, J. Pure Appl. Algebra \textbf{69} (1990) 1--19.
	
	\bibitem{AAZ92} D.~D. Anderson, D.~F. Anderson, and M. Zafrullah: \emph{Factorizations in integral domains II}, J. Algebra \textbf{152} (1992) 78--93.

	\bibitem{AJ15} D.~D. Anderson and J.~R. Juett: \emph{Long length functions}, J. Algebra \textbf{426} (2015) 327--343.
	
	\bibitem{AKZ} D.~D. Anderson, D.~J. Kwak, and M. Zafrullah: \emph{On agreeable domains}, Comm. Algebra \textbf{23} (1995) 4861--4883.
	
	\bibitem{AM96} D.~D. Anderson and B. Mullins: \emph{Finite factorization domains}, Proc. Amer. Math. Soc. \textbf{124} (1996) 389--396.

	\bibitem{AG22} D. F. Anderson and F. Gotti: \emph{Bounded and finite factorization domains}. In: Rings, Monoids, and Module Theory (Eds. A. Badawi and J. Coykendall) pp. 7--57. Springer Proceedings in Mathematics \& Statistics, Vol. \textbf{382}, Singapore, 2022.

	\bibitem{AS75} J.~T. Arnold and P.~B Sheldon: \emph{Integral domains that satisfy Gauss Lemma}, Michigan Math. J. \textbf{22} (1975) 39--51.

	\bibitem{BCG21} N. R. Baeth, S. T. Chapman, and F. Gotti: \emph{Bi-atomic classes of positive semirings}, Semigroup Forum \textbf{103} (2021) 1--23.

	\bibitem{BD76} R.~A. Beauregard and D.~E. Dobbs: \emph{On a class of Archimedean integral domains}, Can. J. Math. \textbf{28} (1976) 365--375.

	\bibitem{BR76} J. Brewer and E.~A. Rutter: \emph{$D+M$ constructions with general overrings}, Michigan Math. J. \textbf{23} (1976) 33--42.

	\bibitem{gwC15} G.~W. Chang: \emph{Star operations on Pr\"ufer $v$-multiplication domains}, J. Commut. Algebra \textbf{7} (2015) 523--543.

	\bibitem{CDZ} G.~W. Chang, T. Dumitrescu, and M. Zafrullah: \emph{$t$-splitting sets in integral domains}, J. Pure Appl. Algebra \textbf{187} (2004) 71--86.
	
	\bibitem{CCGS21} S. T. Chapman, J. Coykendall, F. Gotti, and W. Smith: \emph{Length-factoriality in commutative monoids and integral domains}, J. Algebra \textbf{578} (2021) 186--212.
	
	\bibitem{pC17} P.~L. Clark: \emph{The Euclidean criterion for irreducibles}, Amer. Math. Monthly \textbf{124} (2017) 198--216.
	
	\bibitem{pC68} P.~M. Cohn: \emph{Bezout rings and their subrings}, Proc. Cambridge Philos. Soc. \textbf{64} (1968) 251--264.

	\bibitem{CG22} J. Correa-Morris and F. Gotti: \emph{On the additive structure of algebraic valuations of polynomial semirings}, J. Pure Appl. Algebra \textbf{226} (2022) 107104.
	
	\bibitem{CMZ78} D. Costa, J.~L. Mott, and M. Zafrullah: \emph{The construction $D + XD_S[X]$}, J. Algebra \textbf{53} (1978) 423--439.

	\bibitem{CG19} J. Coykendall and F. Gotti: \emph{On the atomicity of monoid algebras}, J. Algebra \textbf{539} (2019) 138--151.

	\bibitem{EK18} S. Eftekhari and M.~R. Khorsandi: \emph{MCD-finite domains and ascent of IDF-property in polynomial extensions}, Comm. Algebra \textbf{46} (2018) 3865--3872.

	\bibitem{jE19} J. Elliott: \emph{Rings, Modules, and Closure Operations}, Springer Monographs in Mathematics, Springer, Cham 2019.

	\bibitem{GGT21} A. Geroldinger, F. Gotti, and S. Tringali: \emph{On strongly primary monoids, with a focus on Puiseux monoids}, J. Algebra \textbf{567} (2021) 310--345.

	\bibitem{GH06} A.~Geroldinger and F.~Halter-Koch: \emph{Non-unique Factorizations: Algebraic, Combinatorial and Analytic Theory}, Pure and Applied Mathematics Vol. 278, Chapman \& Hall/CRC, Boca Raton, 2006.
	
	\bibitem{rG84} R. Gilmer: \emph{Commutative Semigroup Rings}, The University of Chicago Press, 1984.
	
	\bibitem{rG72} R. Gilmer: \emph{Multiplicative Ideal Theory}, Pure and Applied Mathematics, vol. \textbf{12}, Marcel-Dekker, New York, 1972.
	
	\bibitem{rG68} R. Gilmer: \emph{Multiplicative Ideal Theory}, Queen's Papers in Pure and Applied Mathematics, No. 12, Queen's Univ. Press, Kingston, Ontario, 1968.
	
	\bibitem{GMZ} R. Gilmer, J. Mott, and M. Zafrullah: \emph{On $t$-invertibility and comparability}, Commutative Ring Theory (Eds. P.-J. Cahen, D. Costa, M.	Fontana, and S.-E. Kabbaj), Marcel Dekker, New York, 1994, 141--150.
	
	\bibitem{GP74} R.~Gilmer and T.~Parker: \emph{Divisibility properties of semigroup rings}, Michigan Math. J. \textbf{21} (1974) 65--86.
	
	\bibitem{fG19} F.~Gotti: \emph{Increasing positive monoids of ordered fields are FF-monoids}, J. Algebra \textbf{518} (2019) 40--56.

	\bibitem{fG22} F. Gotti: \emph{On semigroup algebras with rational exponents}, Comm. Algebra \textbf{50} (2022) 3--18.

	\bibitem{GL22} F. Gotti and B. Li: \emph{Divisibility in rings of integer-valued polynomials}, New York J. Math \textbf{28} (2022) 117--139.

	\bibitem{GW75} A.~Grams and H.~Warner: \emph{Irreducible divisors in domains of finite character}, Duke Math. J. \textbf{42} (1975) 271--284.

	\bibitem{HMM} E. Houston, S. Malik, and J. Mott: \emph{Characterizations of $\ast$-multiplication domains}, Canad. Math. Bull. \textbf{27} (1984) 48--52.

	\bibitem{HZ t-inv} E. Houston and M. Zafrullah: \emph{On t-invertibility II}, Comm. Algebra \textbf{17} (1989) 1955--1969.

	\bibitem{hK98} H. Kim: \emph{Factorization in monoid domains}. PhD Dissertation, The University of Tennessee, Knoxville, 1998.

 	\bibitem{MO06} P. Malcolmson and F. Okoh: \emph{A class of integral domains between factorial domains and IDF-domains}, Houston J. Math. \textbf{32} (2006) 399--421.
 
	\bibitem{MO09} P. Malcolmson and F. Okoh: \emph{Polynomial extensions of idf-domains and of idpf-domains}, Proc. Amer. Math. Soc. \textbf{137} (2009) 431--437.
	
	\bibitem{hT72} H.~T. Tang: \emph{Gauss' Lemma}, Proc. Amer. Math. Soc. \textbf{35} (1972) 372--376.
	
	\bibitem{mZ87} M. Zafrullah: \emph{On a property of pre-Schreier domains}, Comm. Algebra \textbf{15} (1987) 1895--1920.
	
	\bibitem{mZ78} M. Zafrullah: \emph{On finite conductor domains}, Manuscripta Math. \textbf{24} (1978) 191--204.
	
	\bibitem{mZ17} M. Zafrullah: \emph{Question HD 1704}. https://lohar.com/mithelpdesk/hd1704.pdf
	
	\bibitem{mZ22} M. Zafrullah: \emph{Question 2202}. https://lohar.com/mithelpdesk/hd2202.pdf
	
	\bibitem{Z poly} M. Zafrullah: \emph{Some polynomial characterizations of Pr\"{u}fer $v$-multiplication domains}, J. Pure Appl. Algebra \textbf{32} (1984) 231--237.
	
	\bibitem{Z WB} M. Zafrullah: \emph{Well behaved prime t-ideals}, J . Pure Appl. Algebra \textbf{65} (1990) 199--207.

\end{thebibliography}
\end{document}